\documentclass[11pt]{article}
\usepackage{epigamath}

\setpapertype{A4}

\usepackage[english]{babel}

\usepackage[dvips]{graphicx}     

\title{Some remarks on regular foliations with numerically trivial canonical class}
\titlemark{Some remarks on regular foliations with numerically trivial canonical class}
\author{St\'ephane Druel}
\authoraddresses{
\authordata{St\'ephane Druel}{\firstname{St\'ephane} \lastname{Druel}\\
\institution{Institut Fourier, UMR 5582 du CNRS, Universit\'e Grenoble Alpes, CS 40700, 38058 Grenoble cedex 9, France}\\
\email{stephane.druel@univ-grenoble-alpes.fr}}
}
\authormark{S. Druel}
\date{\vspace{-5ex}} 
\journal{\'Epijournal de G\'eom\'etrie Alg\'ebrique} 
\acceptation{Received by the Editors on October 5, 2016, and in final form on April 7, 2017. \\ Accepted on May 12, 2017.}

\newcommand{\articlereference}{Volume 1 (2017), Article Nr. 4}

\acknowledgement{The author was partially supported by the ALKAGE project (ERC grant Nr 670846, 2015--2020).}

\usepackage[matrix,arrow]{xy}
\xyoption{all}
\usepackage{amscd,amssymb,amsfonts,amsmath}

\usepackage{graphics}
\usepackage{epsfig}
\usepackage{mathrsfs}

\newcommand{\codim}{\textup{codim}}

\let \cedilla =\c
\renewcommand{\c}[0]{{\mathbb C}}

\newcommand{\N}{\textup{N}}

\newcommand{\NE}{\overline{\textup{NE}}}
\newcommand{\Hol}{\textup{Hol}}

\renewcommand{\phi}{\varphi}

\newcommand{\lra}{\longrightarrow}

\newcommand{\wt}{\widetilde}

\renewcommand{\le}{\leqslant}
\renewcommand{\ge}{\geqslant}

\newcommand{\sA}{\mathscr{A}}

\newcommand{\sE}{\mathscr{E}}

\newcommand{\sG}{\mathscr{G}}
\newcommand{\sH}{\mathscr{H}}

\newcommand{\sK}{\mathscr{K}}
\newcommand{\sL}{\mathscr{L}}

\newcommand{\sN}{\mathscr{N}}
\newcommand{\sO}{\mathscr{O}}

\newcommand{\sQ}{\mathscr{Q}}

\newtheorem{thm}{Theorem}[section]

\newtheorem{question}[thm]{Question}
\newtheorem{lemma}[thm]{Lemma}
\newtheorem{cor}[thm]{Corollary}

\newtheorem{prop}[thm]{Proposition}

\newtheorem{conj}[thm]{Conjecture}

\newenvironment{thm*}[1]{\bigskip\noindent {\bf Theorem }{\rm (#1)}{\bf .} \it}{\bigskip \rm}
\newtheorem{defn}[thm]{Definition}

\newtheorem{say}[thm]{\!\!}
\newtheorem{exmp}[thm]{Example}

\newtheorem{notation}[thm]{Notation}

\newtheorem{defn-thm}[thm]{Definition-Theorem} 
\newtheorem{defn-lemma}[thm]{Definition-Lemma}

\newtheorem{rem}[thm]{Remark}

\def\factor#1.#2.{\left. \raise 2pt\hbox{$#1$} \right/\hskip -2pt\raise -2pt\hbox{$#2$}}

\begin{document}


\maketitle



\begin{prelims}

\vspace{-0.1cm}

\def\abstractname{Abstract}
\abstract{In this article, we first describe codimension two regular foliations with numerically trivial canonical class on 
complex projective manifolds whose canonical class is not numerically effective. 
Building on a recent algebraicity criterion for leaves of algebraic foliations, we then address
regular foliations of small rank with numerically trivial canonical class on 
complex projective manifolds whose canonical class is pseudo-effective.
Finally, we confirm the generalized Bondal conjecture formulated by Beauville in some special cases.}



\vspace{0.05cm}

\languagesection{Fran\cedilla{c}ais}{%

\textbf{Titre. Quelques remarques sur les feuilletages r\'eguliers de classe canonique num\'eriquement triviale}
\commentskip
\textbf{R\'esum\'e.}
Dans cet article nous d\'ecrivons tout d'abord les feuilletages de codimension deux de classe canonique num\'eriquement triviale sur des vari\'et\'es projectives complexes dont la classe canonique n'est pas num\'eriquement effective. En nous appuyant sur un crit\`ere d'alg\'ebricit\'e r\'ecent pour les feuilles des feuilletages alg\'ebriques, nous consid\'erons ensuite les feuilletages r\'eguliers de petit rang et de classe canonique num\'eriquement triviale sur des vari\'et\'es projectives complexes dont la classe canonique est pseudo-effective. Finalement, nous confirmons, dans certains cas sp\'ecifiques, la conjecture de Bondal g\'en\'eralis\'ee formul\'ee par Beauville.}

\end{prelims}


\newpage

\tableofcontents

\section{Introduction}

The Beauville-Bogomolov decomposition theorem asserts that any compact K\"ahler manifold with numerically trivial canonical bundle admits an \'etale cover that decomposes into a product of a torus, and irreducible,
simply-connected Calabi-Yau, and holomorphic symplectic manifolds (see \cite{beauville83}).
In \cite{touzet}, Touzet obtained a foliated version of the Beauville-Bogomolov decomposition theorem
for codimension 1 regular foliations with numerically trivial canonical bundle on compact K\"ahler manifolds.
The next statement follows from \cite[Th\'eor\`eme 1.2]{touzet} and \cite[Lemma 5.9]{bobo}.

\begin{thm*}{Touzet}
Let $X$ be a complex projective manifold, and let
$\sG \subset T_X$ be a regular codimension $1$ foliation on $X$ with $K_\sG\equiv 0$.
Then one of the following holds.
\begin{enumerate}[\rm (1)]
\item There exists a $\mathbb{P}^1$-bundle structure $\phi\colon X \to Y$ onto a complex projective manifold $Y$ with $K_Y\equiv 0$, and $\sG$ induces a flat holomorphic connection on $\phi$.
\item There exists an abelian variety $A$ as well as a simply connected projective manifold $Y$ with $K_Y\equiv 0$, and a finite \'etale cover $f: A \times Y \to X$ such that $f^{-1}\sG$ is the pull-back of a codimenion 1 linear foliation on $A$.
\item There exists a smooth complete curve $B$ of genus at least $2$ as well as a complex projective manifold $Y$ with $K_Y\equiv 0$, and a finite \'etale cover $f \colon B \times Y \to X$ such that $f^{-1}\sG$ is
induced by the projection morphism $B \times Y \to B$.
\end{enumerate}
\end{thm*}

\noindent Pereira and Touzet addressed regular foliations $\sG$ on complex projective manifolds with $c_1(\sG)\equiv 0$ and $c_2(\sG)\equiv 0$ in \cite{pereira_touzet}. They show the following (see \cite[Theorem C]{pereira_touzet}).

\begin{thm*}{Pereira--Touzet}
Let $X$ be a complex projective manifold, and let
$\sG \subset T_X$ be a regular foliation on $X$ with $c_1(\sG)\equiv 0$ and $c_2(\sG)\equiv 0$.
Then, there exists a finite \'etale cover $f\colon Z \to X$ as well as a dense open subset $Z^\circ \subset Z$, and a proper morphism $\phi^\circ\colon Z^\circ \to Y^\circ$ whose fibers are abelian varieties such that ${f^{-1}\sG}_{|Z^\circ} \subset T_{Z^\circ/Y^\circ}$ and 
such that ${f^{-1}\sG}_{|Z^\circ}$ induces linear foliations on the fibers of~$\phi^\circ$.
\end{thm*}

\noindent They also show by way of example that $\phi^\circ$ may not be a locally trivial fibration for the Euclidean topology (see \cite[Example 4.2]{pereira_touzet}). On the other hand, this phenomenon does not occur if $\sG$ has codimension $\le 2$. More precisely, they prove the following (see \cite[Theorem D]{pereira_touzet}). 

\begin{thm*}{Pereira--Touzet}
Let $X$ be a complex projective manifold, and let
$\sG \subset T_X$ be a regular codimension $2$ foliation on $X$ with $c_1(\sG)\equiv 0$ and $c_2(\sG)\equiv 0$.
Then, there exists an abelian variety $A$ as well as a complex projective manifold $Y$, 
and a finite \'etale cover $f: A \times Y \to X$ such that $f^{-1}\sG$ is the pull-back of a linear foliation on $A$.
\end{thm*}

\medskip

In this article, we first address regular codimension 2 foliations with numerically trivial canonical class on 
complex projective manifolds $X$ with $K_X$ not nef.

\begin{thm}\label{thmintro:cd2_non_nef}
Let $X$ be a complex projective manifold, and let
$\sG \subset T_X$ be a regular codimension 2 foliation on $X$ with $K_\sG\equiv 0$.
Suppose that $K_X$ is not nef.
Then one of the following holds.
\begin{enumerate}[\rm (1)]
\item There exist a complex projective manifold $Y$, a regular codimension 2 foliation $\sG_Y$ on $Y$ with 
$K_{\sG_Y}\equiv 0$, and a compact leaf $Z \subset Y$ of $\sG_Y$ such that $X$ 
is the blow-up of $Y$ along $Z$. Moreover, $\sG$ is the pull-back of $\sG_Y$.
\item There exists a conic bundle structure $\phi\colon X \to Y$ on $X$ onto a complex projective manifold $Y$,
$\sG$ is generically transverse to $\phi$ and induces a 
regular codimension 1 foliation $\sG_Y$ on $Y$ with 
$K_{\sG_Y}\equiv 0$. Moreover, 
the discriminant locus of $\phi$ is either empty or a union of leaves of $\sG_Y$.
\item There exists a smooth morphism $\phi\colon X \to Y$ onto a complex projective manifold $Y$ 
of dimension $\dim Y = \dim X -2$
with 
$K_Y\equiv 0$, and $\sG$ yields a flat holomorphic connection on $\phi$. 
\end{enumerate}
\end{thm}

In \cite{pereira_touzet}, the authors show that if $\sG$ is a regular foliation with $c_1(\sG)\equiv 0$ on a complex projective manifold $X$ with $K_X$ pseudo-effective, then $\sG$ is polystable with respect to any polarization. They also conjecture that if moreover $\sG$ is stable with respect to some polarization and $c_2(\sG)\not\equiv 0$, then $\sG$ is algebraically integrable (see 
\cite[Section 6.5]{pereira_touzet}). We recently confirmed this conjecture in some special cases in  
\cite{bobo}. Building on these results, we obtain the following.

\begin{thm}\label{thmintro:cd2_psef}
Let $X$ be a complex projective manifold, and let
$\sG \subset T_X$ be a regular foliation on $X$ with $K_\sG\equiv 0$.
Suppose that $K_X$ is pseudo-effective. Suppose furthermore that $\sG$ has rank at most $3$.
Then there exist complex projective manifolds $Y$ and $F$, a finite \'etale cover $f \colon Y \times F \to X$, and a regular foliation $\sH$ on $Y$ with $c_1(\sH)\equiv 0$ and $c_2(\sH)\equiv 0$
such that $f^{-1}\sG$ is the pull-back of $\sH$.
\end{thm}

Notice that Touzet addressed codimension 2 regular foliations $\sG$ 
with $c_1(\sG)\equiv 0$ on compact K\"ahler manifolds with seminegative Ricci curvature in \cite{touzet_cd2}. See \cite[Th\'eor\`eme 1.5]{touzet_cd2} for a precise statement.

\medskip

Finally, we confirm the generalized Bondal conjecture formulated by Beauville (\cite[Conjecture 5]{beauville_problem_list}) in some special cases. Recall that a holomorphic Poisson structure on a complex manifold $X$ is a bivector field 
$\tau \in H^0(X,\wedge^2T_X)$, such that the bracket $\{f,g\}:=\langle \tau ,df\wedge dg\rangle$ defines a Lie algebra structure on $\sO_X$. A Poisson structure defines a skew-symmetric  map $\tau ^\sharp:\Omega ^1_X\rightarrow T_X$; the rank of $\tau $ at a point $x\in X$ is the rank of $\tau ^\sharp(x)$. 
Let $r$ be the generic rank of $\tau$, and set 
$X_r :=\{x\in X\ |\ \textup{rk }\tau^\sharp (x)=r\}$.

\begin{thm}\label{thmintro:generalized_bondal_conjecture}
Let $(X,\tau )$ be a projective Poisson manifold. Suppose that $\tau$ has generic 
rank $r \ge \dim X -2$.
If the  degeneracy locus $X\setminus X_{r}$ of $\tau$ is non-empty, then it has a component of dimension $>r-2$.
\end{thm}

\medskip

\noindent{\bf Outline of the paper.}
In section 2, we review basic definitions and results about foliations. 
Sections 3 and 4 consist of technical preparations. In section 3, we review standard facts about (partial) connections and give some applications. Section 4 presents a criterion that guarantees that a given
foliation is projectable under a morphism. We then study projectable foliations under some Fano-Mori contractions. With these preparations at hand, the proof of Theorem \ref{thmintro:cd2_non_nef} which we give in Section 5 becomes reasonably short. Section 6 is devoted to the proof of Theorem \ref{thmintro:cd2_psef}. 
The proof relies on a global version of Reeb stability theorem for regular foliations with numerically trivial canonical class on complex projective manifolds with pseudo-effective canonical divisor. 
Finally, we prove Theorem \ref{thmintro:generalized_bondal_conjecture}
in section 7.

\

\noindent {\bf Notation and global conventions.}
We work over the field $\mathbb{C}$ of complex numbers. Varieties are always assumed to be irreducible. We denote by $X_{\textup{reg}}$ the smooth locus of a variety $X$.

The word ``stable'' will always mean ``slope-stable with respect to a
given polarization''. Ditto for semistability.

\

\noindent {\bf Acknowledgements.} We would like to thank Pierre-Emmanuel Chaput,
Jorge V. Pereira and Fr\'ed\'eric Touzet for helpful discussions. We also thank the referee for their thoughtful suggestions on how to improve the presentation of some of the results.

\section{Foliations}

In this section we recall basic facts concerning foliations.

\begin{say}[Foliations] \rm
A \emph{foliation} on a normal variety $X$ is a coherent subsheaf $\sG\subset T_X$ such that
\begin{itemize}
	\item $\sG$ is closed under the Lie bracket, and
	\item $\sG$ is saturated in $T_X$. In other words, the quotient $T_X / \sG$ is torsion free.
\end{itemize}
The \emph{rank} $r$ of $\sG$ is the generic rank of $\sG$.
The \emph{codimension} of $\sG$ is defined as $q:=\dim X-r$. 
Let $U \subset X_{\textup{reg}}$ be the open set where $\sG_{|X_{\textup{reg}}}$ is a subbundle of $T_{X_{\textup{reg}}}$. 
The \textit{singular locus} of $\sG$ is defined
as $Z(\sG):=X \setminus U$.
We say that $\sG$ is \emph{regular} if $U=X_{\textup{reg}}=X$. 

A \emph{leaf} of $\sG$ is a connected, locally closed holomorphic submanifold $L \subset U$ such that
$T_L=\sG_{|L}$. A leaf is called \emph{algebraic} if it is open in its Zariski closure.

The \emph{canonical class} $K_\sG$ of $\sG$ is any Weil divisor on $X$ such that $\sO_X(-K_\sG)\cong \det(\sG):=\big(\wedge^r \sG\big)^{**}$.
\end{say}

We will use the following notation.

\begin{notation} \rm
Let $\phi\colon X \to Y$ be a dominant morphism of normal varieties, and let $D$ be a Weil $\mathbb{Q}$-divisor on $Y$. 
Suppose that $\phi$ is equidimensional.
The \emph{pull-back} $\phi^*D$ of $D$ is defined as follows. We define 
$\phi^*D$ to be the unique $\mathbb{Q}$-divisor on $X$ whose restriction to 
$\phi^{-1}(Y_{\textup{reg}})$ is $(\phi_{|Y_\textup{reg}})^*D_{|Y_\textup{reg}}$. This construction agrees with the usual pull-back if $D$ itself is $\mathbb{Q}$-Cartier.
\end{notation}

\begin{notation} \rm
Let $\phi \colon X \to Y$ be a dominant morphism of normal varieties. Suppose that 
$\phi$ is equidimensional. 
Write
$K_{X/Y}:=K_X-\phi^*K_Y$. We refer to it as the \emph{relative canonical divisor of $X$ over $Y$}.
Set $R(\phi)=\sum_{D} \Big((\phi)^*D-{\big((\phi)^*D\big)}_{\textup{red}}\Big)$
where $D$ runs through all prime divisors on $Y$. We refer to it
as the \emph{ramification divisor of $\phi$}.
\end{notation}

\begin{exmp}\label{exmp:can_div_fol_morphism} \rm
Let $\phi\colon X \to Y$ be a dominant morphism of normal varieties.
Suppose that $\phi$ is equidimensional, and let $\sG$ be the foliation on $X$ induced by $\phi$.
A straightforward computation shows that $K_\sG\sim K_{X/Y}-R(\phi)$, where $R(\phi)$ denotes the ramification divisor of $\phi$.
\end{exmp}

\begin{say}[Foliations described as pullbacks] \rm
Let $X$ and $Y$ be normal varieties, and let $\phi\colon X \to Y$ be a dominant morphism. Let 
$\sG$ be a foliation on $Y$. The \emph{pullback $\phi^{-1}\sG$ of $\sG$ under $\phi$} is 
defined as follows. We define $\phi^{-1}\sG$ to be the unique foliation on $X$ whose restriction to 
$\phi^{-1}(Y_{\textup{reg}})$ is the saturation of 
$$\big(d(\phi_{|\phi^{-1}(Y_{\textup{reg}})})\big)^{-1}\big((\phi_{|\phi^{-1}(Y_{\textup{reg}})})^*\sG_{|Y_{\textup{reg}}}\big)$$ in ${T_X}_{|\phi^{-1}(Y_{\textup{reg}})}$, where 
$d(\phi_{|\phi^{-1}(Y_{\textup{reg}})})\colon {T_X}_{|\phi^{-1}(Y_{\textup{reg}})} \to (\phi_{|\phi^{-1}(Y_{\textup{reg}})})^*T_{Y_{\textup{reg}}}$ denotes the natural map.
\end{say}

A straightforward computation shows the following.

\begin{lemma}\label{lemma:pull_back_fol_and_finite_cover}
Let $f \colon Z\to X$ be a finite cover of normal varieties, and let $\sG$ be a foliation on $X$. Suppose that the branch locus of $f$ is $\sG$-invariant. Then $f^{-1}\sG \cong (f^*\sG)^{**}$.
\end{lemma}

\begin{say}[Projection of foliations]\label{say:morphisms_and_foliations} \rm 
Let $X$ and $Y$ be complex manifolds, and let $\sG \subset T_X$ be a foliation 
of rank $r$ on $X$. Let also
$\phi\colon X \to Y$ be a dominant morphism. 
We say that $\sG$ is \emph{projectable under $\phi$} if, for a general point $y \in Y$, $d_x\phi(\sG_x)$ is independent of the choice of $x \in \phi^{-1}(y)$ and $\dim d_x\phi (\sG_x) = r$, where $d_x\phi\colon T_xX \to T_yY$ denotes the natural map. 
We claim that $\sG$ induces a foliation $\sG_Y$ of rank $r$ on $Y$. Indeed, we may assume without loss of generality that $\phi$ is smooth. Let $y \in Y$ be a general point, and let $s \colon W \to X$ be a section of $\phi$ over some open neighborhood $W$ of $y$ with respect to the Euclidean topology.
Since $\sG$ is projectable, $\sG_{|s(W)}\subset T_{s(W)}\subset (T_X)_{|s(W)}$, and
we must have $\sG_{|\phi^{-1}(W)} \cong (\phi_{|\phi^{-1}(W)})^*(s^*\sG)$. This immediately implies that $s^*\sG \subset T_W$
is independent of $s$, and that it is stable under the Lie bracket. Let $\sG_Y$ be the foliation on $Y$ such that 
$(\sG_Y)_{|W}:=s^*\sG$.
We will refer to $\sG_Y$ as the \textit{projection of $\sG$ under $\phi$}. 

Suppose from now on that $\phi$ is a proper morphism with connected fibers. Suppose furthermore that
$\det(\sG)\cong \phi^*\sL$ for some line bundle $\sL$ on $Y$.

The $r$-th wedge product of the inclusion $\sG \subset T_X$ gives rise to a non-zero 
global section $\tau \in H^0(X,\wedge^r T_X\otimes \phi^*\sL^*)$ mapping to a possibly zero global section $\tau_Y \in H^0(Y,\wedge^r T_Y\otimes \sL^*)$ under
the natural map
$$H^0(X,\wedge^r T_X\otimes \phi^*\sL^*)\lra H^0\big(X,\wedge^r (\phi^*T_Y)\otimes\phi^*\sL^*\big)\cong H^0(Y,\wedge^r T_Y\otimes \sL^*).$$ 
Under these assumptions, $\sG$ is projectable under $\phi$ if and only if 
$\tau_Y \neq 0$. Equivalently, $\sG$ is projectable under $\phi$ if and only if there exists $x \in X$ such that $\dim d_x\phi (\sG_x) = r$.
Let $Z(\tau_Y)$ be the zero set of $\tau_Y$. Then 
$\phi\big( Z(\sG)\big) \subset Z(\tau_Y)$, and 
conversely, if $x \in X\setminus Z(\sG)$ and $\dim d_x\phi (\sG_x) = r$, then $\phi(x) \in Y\setminus Z(\tau_Y)$.

Finally, suppose that $\tau_Y\neq 0$, and let $\sG_Y$ be the projection of $\sG$ under $\phi$.
We claim that $-K_{\sG_Y} - c_1(\sL)$ is effective. Indeed, 
let $U \subset Y$ be the Zariski open set where $\sG_Y$ is a subbundle of $T_Y$. The natural map
$d\phi\colon T_X \to \phi^*T_Y$ induces a morphism $\sG_{|\phi^{-1}(U)} \to (\phi_{|\phi^{-1}(U)})^*(\sG_Y)_{|U}$
which is generically injective. In particular, we must have $(\tau_Y)_{|U} \in H^0\big(U,(\det(\sG_Y)\otimes \sL^*)_{|U}\big)$. 
Since $Y \setminus U$ has codimension at least two, we conclude that
$-K_{\sG_Y} - c_1(\sL)$ is effective, proving our claim.
Suppose moreover that $-K_{\sG_Y} \equiv c_1(\sL)$. Then $(\det(\sG_Y)\otimes \sL^*)_{|U}\cong \sO_U$ and $Z(\tau_Y) \subset Y \setminus U$. It follows that $Z(\tau_Y)$ has codimension at least two, and hence
$Z(\sG_Y)=Z(\tau_Y)$. In particular, we must have $\phi\big(Z(\sG)\big) \subset Z(\tau_Y)=Z(\sG_Y)$.
\end{say}

\begin{rem}\label{remark:fano_mori_contraction_line_bundle} \rm
Let $X$ be a complex projective manifold, and let $\sG \subset T_X$ be a foliation on $X$. Let also
$\phi\colon X \to Y$ be a proper morphism with connected fibers onto a normal projective variety $Y$. 
Suppose that $-K_X$ is $\phi$-ample. If $K_\sG$ is $\phi$-numerically trivial, then there exists a line bundle $\sL$ on $Y$ such that
$\det(\sG)\cong \phi^*\sL$ (see \cite[Lemma 3.2.5]{kmm}).
\end{rem}

\section{Bott (partial) connection and applications}

In this section we provide technical tools for the proof of the main result.

\begin{say}[Bott connection]\label{partial_connection} \rm
Let $X$ be a complex manifold, let $\sG\subset T_X$ be a regular foliation, and set $\sN=T_X/\sG$. Let $p\colon T_X\to \sN$ denotes the natural projection. For sections $U$ of $\sN$, $T$ of $T_X$, and $V$ of $\sG$ over some open subset of $X$ with 
$U=p(T)$, set $D_V U=p([V,U])$. This expression is well-defined,
$\sO_X$-linear in $V$ and satisfies the Leibnitz rule 
$D_V(fU)=fD_V U+(Vf)U$ so that $D$ is a $\sG$-connection on $\sN$
(see \cite{baum_bott70}). We refer to it as the \textit{Bott connection} on $\sN$.

\end{say}

\begin{lemma}\label{lemma:normal_bundle_restricted_leaf}
Let $X$ be a complex manifold, and let $\sG\subset T_X$ be a regular foliation 
with normal bundle $\sN=T_X/\sG$. Let $f\colon Z \to X$ be a compact manifold, and 
suppose that $f(Z)$ is tangent to $\sG$. Then $f^*\sN$ admits a flat holomorphic connection. In particular,  Chern classes of $f^*\sN$ vanish.
\end{lemma}

\begin{proof}
One readily checks that the (partial) Bott connection on $\sN$ gives a flat holomorphic connections on
$f^*\sN$.
\hfill $\Box$
\end{proof}

\begin{say}[The Atiyah class of a locally free sheaf]  \label{say:atiyah} \rm
Let $X$ be a complex manifold, and let $\sE$ be a locally free sheaf on $X$.
Let $J_X^1(\sE)$ be the sheaf of $1$-jets of $\sE$. 
As a sheaf of abelian groups on $X$, 
$J_X^1(\sE)\cong \sE\oplus (\Omega_X^1\otimes\sE)$,  and the $\sO_X$-module structure is given 
by $f(e,\alpha)=(fe,f\alpha-df\otimes e)$, where $f$, $e$ and $\alpha$ are 
local sections of $\sO_X$, $\sE$ and $\Omega_X^1\otimes\sE$, respectively. 
The \emph{Atiyah class} of $\sE$ 
(see \cite{atiyah57})
is defined to be the element 
$at(\sE)\in H^1\big(X,\Omega_X^1\otimes\sE\hspace{-0.07cm}\textit{nd}_{\sO_X}(\sE)\big)$ 
corresponding to the extension 
$$
0\to \Omega_X^1\otimes\sE \to J_X^1(\sE) \to \sE \to 0.
$$
\end{say}

The next result follows from the proof \cite[Proposition 3.3]{baum_bott70} and \cite[Corollary 3.4]{baum_bott70} (see also \cite[Lemma 6.3]{codim_1_del_pezzo_fols}). 

\begin{lemma}\label{lemma:atiyah_class_foliation}
Let $X$ be a complex manifold, and let $\sG\subset T_X$ be a regular foliation. Set $\sN:=T_X/\sG$.
Then $at(\sN)\in H^1\big(X,\Omega^1_X\otimes\sE\hspace{-0.07cm}\textit{nd}_{\sO_X}(\sN)\big)$ lies in the 
image of the natural map 
$$H^1\big(X,\sN^*\otimes\sE\hspace{-0.07cm}\textit{nd}_{\sO_X}(\sN)\big)\to
H^1\big(X,\Omega^1_X\otimes\sE\hspace{-0.07cm}\textit{nd}_{\sO_X}(\sN)\big).$$
\end{lemma}

The first part of the statement of Lemma \ref{lemma:vanishing_coh_support} below is \cite[Lemma 6.4]{codim_1_del_pezzo_fols}. 
The proof of the second part of the statement is
similar to that of \cite[Lemma 6.4]{codim_1_del_pezzo_fols}, and so we leave the details to the reader (see also 
\cite[Theorem 1.14]{siu_trautmann}).
Notice that the statement of 
\cite[Lemma 6.4]{codim_1_del_pezzo_fols} is slightly incorrect.

\begin{lemma}\label{lemma:vanishing_coh_support}
Let $X$ be a complex projective manifold, and let $\sE$ be a vector bundle on $X$. Let $U \subset X$ be an open set. Suppose that $X \setminus U$ has codimension $\ge m+2$ for some integer $m\ge 0$. Then
\begin{enumerate}[\rm (1)]
\item $H^i(X,\sE)\cong H^i(U,\sE_{|U})$ for $0\le i \le m$, and
\item the natural map
$H^{m+1}(X,\sE)\to H^{m+1}(U,\sE_{|U})$ is injective.
\end{enumerate}
\end{lemma}

The following result generalizes \cite[Corollary 3.4]{baum_bott70}.
 
\begin{lemma}\label{lemma:vanishing_char_classes}
Let $X$ be a complex projective manifold, and let $\sG \subset T_X$ be a foliation of codimension $q$. 
Set $\sN:=T_X/\sG$. Suppose that the singular locus $Z(\sG)$ of $\sG$ has codimension $\ge q+2$ in $X$. Then   
\begin{enumerate}[\rm (1)]
\item $c_1(\sN)^{q}\in H^q(X,\Omega_X^q)$ lies in the image of the natural map
$H^q\big(X,\det(\sN)^*\big)\to H^q(X,\Omega_X^q)$, and 
\item $c_1(\sN)^{q+1}\equiv 0$.
\end{enumerate}
\end{lemma}

\begin{proof}
Set $U:=X \setminus Z(\sG)$.
By Lemma \ref{lemma:atiyah_class_foliation} and \cite[Theorem 6]{atiyah57},
$c_1(\sN_{|U})\in H^1(U,\Omega_U^1)$ lies in the image of the natural map
$$H^1\big(U,\sN_{|U}^*\big)\to H^1(U,\Omega_U^1).$$ This implies that 
$c_1(\sN_{|U})^{q}\in H^q(U,\Omega_U^q)$ lies in the image of the natural map
$$H^q\big(U,\det(\sN_{|U})^*\big)\to H^q(U,\Omega_U^q),$$ and that $c_1(\sN_{|U})^{q+1}\equiv 0$.
On the other hand, by Lemma \ref{lemma:vanishing_coh_support} above, the restriction map $\det(\sN^*) \to \det(\sN^*_{|U})$ induces
an isomorphism $H^q\big(X,\det(\sN^*)\big)\cong H^q\big(U,\det(\sN^*_{|U})\big)$ as well as an injective map
$H^{q+1}\big(X,\det(\sN^*)\big)\to H^{q+1}\big(U,\det(\sN^*_{|U})\big)$. The lemma then follows easily.
\hfill $\Box$
\end{proof}

Lemma \ref{lemma:dimension_fiber_contraction} below will prove to be crucial. Before starting with the proof, we recall the basic facts concerning Mori theory.

\begin{say}[Fano-Mori contractions] \rm
Let $X$ be a complex projective manifold and consider the finite dimensional $\mathbb{R}$-vector space 
$$\N_1(X)=\big(\{1-\text{cycles}\}/\equiv\big)\otimes\mathbb{R},$$
where $\equiv$ denotes numerical equivalence. The \emph{Mori cone} of $X$ is the closure $\NE(X)\subset \N_1(X)$ of the cone 
spanned by classes of effective curves. 
Every face $V$ of 
$\NE(X)$ contained in $\{z \in \N_1(X) \, | \, K_X\cdot z<0\}$
corresponds to a surjective morphism with connected fibers $\phi\colon X \to Y$ onto a normal projective variety, which is called a \emph{Fano-Mori contraction}. The morphism $\phi$ contracts precisely those curves on $X$ with class in $V$ (see \cite[Theorem 3.2.1]{kmm}). 
Conversely, any morphism  $\phi\colon X \to Y$ with connected fibers onto a normal projective variety such that 
$-K_X$ is $\phi$-ample
arises in this way.
\end{say}

\begin{lemma}\label{lemma:dimension_fiber_contraction}
Let $X$ be a complex projective manifold, and let $\sG \subset T_X$ be a foliation 
of codimension $q$
on $X$. 
Suppose that $K_\sG\equiv 0$ and that the singular locus $Z(\sG)$ of $\sG$ has codimension $\ge q+2$.
Let $\phi\colon X \to Y$ be a Fano-Mori contraction, and let $F$ be an irreducible component 
of some fiber of $\phi$.
Then $\dim F \leq q$. 
If $\dim F > \dim X - \dim Y$ and $\dim F \ge \dim \phi^{-1}(y)$ for all $y \in Y$, then we have $\dim F < q$.
\end{lemma}

\begin{proof}
Set $\sN:=T_X/\sG$. 
By Lemma \ref{lemma:vanishing_char_classes}, we have
$\Big(c_1\big(\det(\sN)\big)_{|F}\Big)^{q+1} \equiv 0$. This implies that $\dim F \le q$ since 
$c_1\big(\det(\sN)_{|F}\big)\equiv {-K_X}_{|F}$ is ample by assumption.

Suppose from now on that $\dim F > \dim X - \dim Y$ and that 
$\dim F \ge \dim \phi^{-1}(y)$ for any $y \in Y$. 
We argue by contradiction and assume that $\dim F = q$.
Let $\nu\colon\widehat{F} \to F$ be a resolution of singularities. Consider the
following commutative diagram:

\centerline{
\xymatrix{
H^q\big(X,\det(\sN)^*\big) \ar[rr]\ar[d] && H^q(X,\Omega_X^q)\ar[d] &\\
H^q\big(F,\det(\sN)^*_{|F})\big)\ar[r] &H^q\big(\widehat{F},\nu^*\det(\sN)^*_{|F})\big)\ar[r] & 
H^q(\widehat{F},\nu^*{\Omega_X^q}_{|F})\ar[r] & H^q(\widehat{F},\Omega_{\widehat{F}}^q).
}
}
\noindent By Lemma \ref{lemma:vanishing_char_classes} again,
$c_1(X)^{q}\in H^q(X,\Omega_X^q)$ lies in the image of the map
$H^q\big(X,\det(\sN)^*\big)\to H^q(X,\Omega_X^q)$. 
Recall from \cite[Lemma 3.2.5]{kmm} that there exists a line bundle $\sL$ on $Y$ such that
$\det(\sG)\cong \phi^*\sL$, or equivalently
$\det(\sN)\cong \sO_X(-K_X)\otimes \phi^*\sL^*$.
By \cite[Proposition 1.7]{andreatta_wisniewski_view}, the group
$H^q\big(F,{\sO(K_X)}_{|F})\big)$ vanishes.
It follows that the group $H^q\big(F,\det(\sN)^*_{|F})\big)$
vanishes as well.
This implies that
$\big({c_1(X)}_{|\widehat{F}}\big)^{q} \equiv 0$, yielding a contradiction as before. This completes the proof of the lemma.
\hfill $\Box$
\end{proof}

\section{Projectable foliations under Fano-Mori contractions}

The proof of the main result relies on a criterion that guarantees that a given
foliation is projectable under a morphism (see \ref{say:morphisms_and_foliations} for this notion), which we establish now.

\begin{prop}\label{prop:projectable_criterion}
Let $X$ be a complex projective manifold, and let $\phi\colon X \to Y$ be a Fano-Mori contraction with $\dim X - \dim Y \le 2$.
Let $\sG \subset T_X$ be a foliation of positive rank $r$ such that $K_\sG$ is $\phi$-numerically trivial, and
let $\tau_Y \in H^0(Y,\wedge^r T_Y\otimes \sL^*)$ be the twisted field as in 
\ref{say:morphisms_and_foliations} (see Remark \ref{remark:fano_mori_contraction_line_bundle}).
Pick $y\in Y$ and set $F:=\phi^{-1}(y)$. Suppose furthermore that 
$F\subset X \setminus Z(\sG)$ and that $\phi$ is smooth along $F$. Then $y\not\in Z(\tau_Y)$.
\end{prop}

Before proving Proposition \ref{prop:projectable_criterion}, we note the following immediate corollary. See Section \ref{say:morphisms_and_foliations} for the notion of projectable foliation.

\begin{cor}\label{cor:projectable_criterion}
Let $X$ be a complex projective manifold, and let $\sG \subset T_X$ be a foliation. 
Let $\phi\colon X \to Y$ be a Fano-Mori contraction with $\dim X - \dim Y \le 2$.
Suppose that $K_\sG\equiv 0$ and that $\dim Z(\sG)\le \dim Y-1$.
Then $\sG$ is projectable under $\phi$.
\end{cor}

\begin{proof}[Proof of Proposition \ref{prop:projectable_criterion}]
Set $m:=\dim Y$, and consider the following commutative diagram:

\centerline{
\xymatrix{
 & T_F \ar@{^{(}->}[d] \\
 \sG_{|F}\ar@{^{(}->}[r]\ar@/_1pc/[dr]^-{\xi} & {T_X}_{|F}\ar@{->>}[d] \\
 & \sN_{F/X} \cong \sO_F^{\oplus m}.
}
}
\noindent  
We argue by contradiction and assume that $\xi$ has generic rank 
$< r$. We denote by $\sK$ its kernel, and by $\sQ$ its image. 
Notice that $\sK=\sG_{|F}\cap T_F \subset T_F$ is a foliation on $F$.
Since 
$\sQ \subset \sN_{F/X} \cong \sO_F^{\oplus m}$, $\det(\sQ)\cong\sO_F(-D)$
for some effective divisor $D$ on $F$. By the adjunction formula, 
$-K_F\sim -{K_X}_{|F}$ is ample. It follows that $\det(\sG_{|F})\cong\sO_F$, and hence 
$\det(\sK)\cong\sO_X(D)$. Suppose first that $D \neq 0$.
Applying \cite[Proposition 7.5]{fano_fols}, we see that there exists a (rational) curve $C$ contained in $F$ and tangent to $\sK$.
But this contradicts Lemma \ref{lemma:normal_bundle_restricted_leaf} since
$\deg\big(\det(\sN)_{|C}\big)=-K_X \cdot C>0$ where $\sN:=T_X/\sG$.
Therefore, we must have $D=0$. Now either $\sK \cong \sO_F$, or $\sK=T_F$ since $\dim F \le 2$ by assumption. 
If $\sK \cong \sO_F$, then by Lemma \ref{lemma:global_vector_field_rat_curve} below, there exists a (rational) curve contained in $F$ and tangent to $\sG$.
If $\sK = T_F$, then any rational curve contained in $F$ is tangent to $\sG$. 
In either case, this contradicts Lemma \ref{lemma:normal_bundle_restricted_leaf} again, completing the proof of the proposition.
\hfill $\Box$
\end{proof}

\begin{lemma}\label{lemma:global_vector_field_rat_curve}
Let $S$ be a smooth projective rational surface, and let $v$ be a nonzero holomorphic global vector field on 
$S$. Then there is a rational curve on $S$ tangent to $v$.
\end{lemma}

\begin{proof}
Denote by $S_0$ a minimal model of $S$, and let $v_0$ be the global holomorphic vector field induced by $v$ on $S_0$. Recall that either $S_0 \cong \mathbb{P}^2$, or $S_0$ is a Hirzebruch surface. If $S_0 \cong \mathbb{P}^2$, then it is well known that there is a line invariant by $v_0$ (see \cite[page 11]{jouanolou}). Suppose that $S_0$ is a Hirzebruch surface. 
If $S \cong \mathbb{P}^1\times \mathbb{P}^1$, then there exist global vector fields $v_1$ and $v_2$ on $\mathbb{P}^1$
such that $v_0=p_1^*v_1+p_2^*v_2$, where $p_1$ and $p_2$ denote the projections on the two factors $\mathbb{P}^1$.
Let $x_1 \in \mathbb{P}^1$ such that $v_1(x_1)=0$. Then the curve $\{x_1\}\times \mathbb{P}^1 \subset S$
is invariant by $v_0$. Finally, suppose that $S \not\cong \mathbb{P}^1\times \mathbb{P}^1$, and denote by 
$C \subset S$ the unique (rational) curve with negative self-intersection. The map 
$$\sO_C \overset{{v_0}_{|C}}{\longrightarrow} {T_S}_{|C} \to \sN_{C/S}$$
vanishes since $\deg(\sN_{C/S})=C\cdot C<0$, and thus $C$ is invariant by $v_0$.
\hfill $\Box$
\end{proof}

\begin{question}\label{question:foliation_rational_curve}
Let $X$ be a Fano manifold, and let $\sG$ be a foliation of positive rank on $X$. Suppose that 
$K_\sG\equiv 0$. Is it true that there exists a curve on $X$ tangent to $\sG$?
\end{question}

\begin{rem} \rm
If Question \ref{question:foliation_rational_curve} has a positive answer, then
Proposition \ref{prop:projectable_criterion} and Corollary \ref{cor:projectable_criterion} hold without restriction on $\dim X - \dim Y$. This is the case when 
$\sG$ is regular by \cite[Theorem 1.2]{druel_fol_fano}.
\end{rem}

\begin{prop}\label{prop:foliation_blow_up}
Let $X$ be a complex projective manifold, and let $\sG \subset T_X$ be a foliation 
of rank $r$ on $X$ with $K_\sG\equiv 0$. 
Suppose that $X$ is the blow-up of a complex projective manifold $Y$ along a codimension $2$ submanifold $B \subset Y$, and denote by $\phi\colon X \to Y$ the natural morphism.
Consider the 
projection $\sG_Y$ of $\sG$ on $Y$.
Then $K_{\sG_Y}\equiv 0$ and $Z(\sG_Y) = \phi\big(Z(\sG)\big)$.
Moreover, if $B \not\subset \phi\big(Z(\sG)\big)$, then $r \le \dim X - 2$ and
$B \setminus Z(\sG_Y)$ is invariant by $\sG_Y$.
\end{prop}

\begin{proof}
Let $\tau_Y \in H^0(Y,\wedge^r T_Y\otimes \sL^*)$ be the twisted field as in 
\ref{say:morphisms_and_foliations} (see Remark \ref{remark:fano_mori_contraction_line_bundle}).

Notice that $K_{\sG_Y}\equiv c_1(\sL)\equiv 0$ since $K_\sG\equiv 0$. Thus, by \ref{say:morphisms_and_foliations}, we have
$\phi\big(Z(\sG)\big) \subset Z(\sG_Y)=Z(\tau_Y)$. If $B \subset \phi\big(Z(\sG)\big)$, then 
$Z(\sG_Y) = \phi\big(Z(\sG)\big)$ since $Z(\sG_Y)$ and $\phi\big(Z(\sG)\big)$ agree away from $B$.

Suppose from now on that $B \not\subset \phi\big(Z(\sG)\big)$.
Set $E:=\textup{Exc}(\phi)$, and let
$C\cong\mathbb{P}^1$ be a fiber of the natural map $E \to B$. 
Set also $y:=\phi(C)$.
Suppose that $C \subset X\setminus Z(\sG)$.
By Lemma \ref{lemma:normal_bundle_restricted_leaf}, $C$ is not tangent to $\sG$.
From \ref{say:morphisms_and_foliations}, we conclude that 
$y \in Y \setminus Z(\sG_Y)$. This shows that
$Z(\sG_Y) \cap B\subset \phi\big(Z(\sG)\big)\cap B$, and thus 
$Z(\sG_Y) = \phi\big(Z(\sG)\big)$.
Since $C$ is not tangent to $\sG$, the inclusion $\sG_{|C}\subset {{T_X}_{|C}}\cong \sO_{\mathbb{P}^1}(2)\oplus\sO_{\mathbb{P}^1}^{\oplus \dim X-2}\oplus\sO_{\mathbb{P}^1}(-1)$ induces an inclusion
$\sG_{|C}\subset \sO_{\mathbb{P}^1}^{\oplus \dim X-2}\oplus\sO_{\mathbb{P}^1}(-1)$.
This implies that $\sG_{|C}\cong \sO_{\mathbb{P}^1}^{\oplus r}\subset {T_E}_{|C}$
since $\det(\sG_{|C})\cong \sO_C$. It follows that
$E$ is invariant by $\sG$, and that $r \le \dim X -2$.
The proposition follows easily.
\hfill $\Box$
\end{proof}

In Lemma \ref{lemma:conic_bundle} below, 
we gather some properties of conic bundles for later reference.
Recall that a \textit{conic bundle structure} on a complex projective manifold $X$ is a surjective morphism
$\phi \colon X \to Y$ onto a complex projective manifold with fibers isomorphic to conics.
Its \textit{discriminant locus} is the set $\Delta:=\{y \in Y \,|\,\textup{such that $X_y:=\phi^{-1}(y)$ is not smooth}\}$.

\begin{lemma}\label{lemma:conic_bundle}
Let $X$ be a complex projective manifold, and let $\phi \colon X \to Y$ be a conic bundle structure on $X$.
Suppose that its discriminant locus
$\Delta$ is non-empty.
Set $\Delta_1:=\{y \in \Delta \,|\,\textup{$X_y$ is reduced}\}$
and $\Delta_2:=\{y \in \Delta \,|\,\textup{$X_y$ is not reduced}\}$. 
Then $\Delta$ is an hypersurface in $Y$, and $\Delta_1 \subset \Delta$ is a dense open set.
Moreover, $\Delta$ has normal crossing singularities in codimension 1, $\Delta_1 \subset \Delta_{\textup{reg}}$, and $\Delta$ is singular along codimension $1$ irreducible components of $\Delta_2$.
Further, if $y$ is any point on $Y$ and $F:=\phi^{-1}(y)$, then the following properties hold in addition.
\begin{enumerate}[\rm (1)]

\item If $y\in \Delta_1$, then $F=C_1\cup C_2$ is the union of $2$ smooth rational curves meeting transversally at a point, $-K_X \cdot C_i=1$
and 
$\sN_{C_i/X}\cong \sO_{\mathbb{P}^1}^{\oplus \dim X -2}\oplus \sO_{\mathbb{P}^1}(-1)$ for 
$i\in \{1,2\}$.

\item If $y\in \Delta_2 \setminus \Delta_{\textup{reg}}$, then $F$ is a smooth rational curve, $-K_X \cdot F=1$, and $\sN_{F/X}\cong \sO_{\mathbb{P}^1}^{\oplus \dim X -3}\oplus \sO_{\mathbb{P}^1}(1)\oplus 
\sO_{\mathbb{P}^1}(-2)$.

\item Let $\widehat{\Delta}$ denotes the Zariski closure in $X$ of the singular locus of $\phi^{-1}(\Delta_1)$.
There exists an open set $\Delta^\circ \subset \Delta$ with complement of codimension $\ge 2$ in $\Delta$
such that the natural map $\widehat{\Delta}\cap \phi^{-1}(\Delta^\circ) \to \Delta^\circ$ identifies with the normalization morphism.
\end{enumerate}
\end{lemma}

\begin{proof}
The same argument used in the proof of \cite[Proposition 1.2]{beauville_prym} shows that $\Delta$ has pure codimension $1$, and that $\Delta_2$ has codimension $\ge 1$ in $\Delta$. It also shows that $\Delta_1$ is smooth, that $\Delta$ has normal crossing singularities in codimension 1, and that $\Delta$ is singular 
along codimension $1$ irreducible components of $\Delta_2$. 

Statements (1) and (2) are due to Ando (see \cite{ando}). Statement (3) follows easily from \cite[Lemme 1.5.2]{beauville_prym}. 
\hfill $\Box$
\end{proof}

\begin{prop}\label{prop:foliation_conic_bundle}
Let $X$ be a complex projective manifold, and let $\sG \subset T_X$ be a foliation 
of rank $r$
on $X$. 
Suppose that $K_\sG\equiv 0$ and that $\dim Z(\sG)\le \dim X -3$.
Let $\phi\colon X \to Y$ be a conic bundle structure on $X$ with discriminant locus $\Delta$.
Then $\sG$ is projectable under $\phi$ by Corollary \ref{cor:projectable_criterion}. Consider the 
projection $\sG_Y$ of $\sG$ on $Y$.
Then $K_{\sG_Y}\equiv 0$ and
$\dim Z(\sG_Y) \le \max\big(\dim Z(\sG),\dim Y- 3\big)$. Moreover, if
$\Delta$ is non-empty, then $r \le \dim X -2$ and 
$\Delta \setminus Z(\sG_Y)$ is invariant by $\sG_Y$.
\end{prop}

\begin{proof}
Let $\tau_Y \in H^0(Y,\wedge^r T_Y\otimes \sL^*)$ be the twisted field as in 
\ref{say:morphisms_and_foliations} (see Remark \ref{remark:fano_mori_contraction_line_bundle}).
By Proposition \ref{prop:projectable_criterion}, $Z(\tau_Y)\setminus\Delta =\phi\big(Z(\sG)\big)\setminus \Delta$.

If $\Delta=\emptyset$, then $Z(\tau_Y)$ has codimension at least $2$ in $Y$, and thus
$Z(\sG_Y)=Z(\tau_Y)$ and
$K_{\sG_Y}\equiv -c_1(\sL)\equiv 0$. Moreover, $\dim Z(\sG_Y)\le \dim Z(\sG)$.

Suppose from now on that $\Delta \neq\emptyset$.
Set $\Delta_1:=\{y \in \Delta \,|\,\textup{$X_y$ is reduced}\}$
and $\Delta_2:=\{y \in \Delta \,|\,\textup{$X_y$ is not reduced}\}$. 

Let $y\in \Delta_1 \setminus \phi \big(Z(\sG)\big)$, and set $F:=\phi^{-1}(y)$. 
Then $F=C_1\cup C_2$ is the union of $2$ smooth rational curves meeting transversally at a point, and 
$-K_X \cdot C_i=1$. 
By Lemma \ref{lemma:normal_bundle_restricted_leaf}, the curve $C_i$ is not tangent to $\sG$, and thus
$y\in Y \setminus Z(\tau_Y)$. This immediately implies that
$Z(\tau_Y)\setminus\Delta_2 =\phi\big(Z(\sG)\big)\setminus \Delta_2$.
As before, we conclude that $Z(\sG_Y)=Z(\tau_Y)$ and that
$K_{\sG_Y}\equiv -c_1(\sL)\equiv 0$.
Since $C_i$ is not tangent to $\sG$,
the inclusion 
$\sG_{|C_i}\subset {T_X}_{|C_i}$ induces an inclusion
$\sG_{|C_i}\subset \sN_{C_i/X}\cong \sO_{\mathbb{P}^1}^{\oplus \dim X -2}\oplus \sO_{\mathbb{P}^1}(-1)$.
It follows that $\sG_{|C_i}\cong \sO_{\mathbb{P}^1}^{\oplus r}$
since $\det(\sG_{|C_i})\cong \sO_{C_i}$. This in turn implies 
that $\phi^{-1}(\Delta)$ is invariant by $\sG$, and that
$\Delta\setminus Z(\sG_Y)$ is invariant by $\sG_Y$.
Moreover, we have $r \le \dim X -2$.

Now, we proceed to show that $\dim Z(\sG_Y) \le \max\big(\dim Z(\sG),\dim Y- 3\big).$
Let $\widehat{\Delta}$ denotes the Zariski closure in $X$ of the singular locus of $\phi^{-1}(\Delta_1)$.
By Lemma \ref{lemma:conic_bundle},
there exists an open set $\Delta^\circ \subset \Delta$ with complement of codimension $\ge 2$ in $\Delta$
such that the natural map $\widehat{\Delta}\cap \phi^{-1}(\Delta^\circ) \to \Delta^\circ$ identifies with the normalization morphism.
By shrinking $\Delta^\circ$ if necessary, we may assume without loss of generality that $\Delta^\circ$ has normal crossing singularities.
This implies that $\widehat{\Delta} \cap \phi^{-1}(\Delta^\circ)$ is smooth, and that 
the restriction of $d_x\phi$ to $T_x\widehat{\Delta}$ is injective.
Since $\hat{\Delta}$ is a union of irreducible components of the singular locus
of $\phi^{-1}(\Delta)$ and since 
$\phi^{-1}(\Delta)$ is invariant by $\sG$, we conclude that $\widehat{\Delta} \setminus Z(\sG)$ is invariant by $\sG$ as well.
This shows that $Z(\sG_Y)\cap\Delta^\circ = \phi\big(Z(\sG)\big)\cap\Delta^\circ$, and hence
$\dim Z(\sG_Y) \le \max\big(\dim Z(\sG),\dim Y- 3\big)$ since 
$\Delta\setminus \Delta^\circ$ has codimension $\ge 3$ in $X$.
This finishes the proof of the proposition. 
\hfill $\Box$
\end{proof}

\begin{prop}\label{prop:foliation_equidim_fibration}
Let $X$ be a complex projective manifold, and let $\sG \subset T_X$ be a foliation 
of rank $r$
on $X$. 
Suppose that $K_\sG\equiv 0$ and that $\dim Z(\sG)\le r -2$.
Let $\phi\colon X \to Y$ be a Fano-Mori contraction with $\dim Y = r$. 
Suppose furthermore that $\sG$ is projectable under $\phi$.
Then 
$Z(\sG)=\emptyset$, $\phi$ is a smooth morphism, and
$\sG$ yields a flat holomorphic connection on $\phi$. In particular, we have $K_Y\equiv 0$.
\end{prop}

\begin{proof}
Set $\sN:=T_X/\sG$.
By \cite[Lemma 3.2.5]{kmm}, there exists a line bundle $\sL$ on $Y$ such that
$\det(\sG)\cong \phi^*\sL$. Notice that $c_1(\sL)\equiv 0$ since $K_\sG\equiv 0$.

\medskip

We first show that $\sL\cong \sO_Y(-K_Y)$.
Set $Y^\circ := Y_{\textup{reg}} \setminus \phi \big(Z(\sG)\big)$, and
notice that $\codim\, Y \setminus Y^\circ \ge 2$.
Set $X^\circ:=\phi^{-1}(Y^\circ)$.
Let $y \in Y^\circ$, and let $F$ be an irreducible component of $\phi^{-1}(y)$.
Denote by $\nu\colon\widehat{F} \to F$ a resolution of singularities.
Notice that $\dim F = \dim X - \dim Y =:q$ by Lemma \ref{lemma:dimension_fiber_contraction} since $\dim Y =r$.

By Lemma \ref{lemma:atiyah_class_foliation}, the Atiyah class $at(\sN_{|X^\circ})\in H^1\big(X^\circ,\Omega^1_{X^\circ}\otimes\sE\hspace{-0.07cm}\textit{nd}_{\sO_{X^\circ}}(\sN_{|X^\circ})\big)$ lies in the image of the natural map 
$$H^1\big(X^\circ,\sN_{|X^\circ}^*\otimes\sE\hspace{-0.07cm}\textit{nd}_{\sO_{X^\circ}}(\sN_{|X^\circ})\big)\to
H^1\big(X^\circ,\Omega^1_{X^\circ}\otimes\sE\hspace{-0.07cm}\textit{nd}_{\sO_{X^\circ}}(\sN_{|X^\circ})\big).$$
In particular, there exists
$\alpha \in H^1\big(X^\circ,\sN_{|X^\circ}^*\big)$ mapping to 
$c_1(\sN_{|X^\circ})\in H^1\big(X^\circ,\Omega^1_{X^\circ}\big)$. Notice that 
$\alpha^q \in H^q\big(X^\circ,\det(\sN_{|X^\circ}^*)\big)\cong 
H^q\big(X^\circ,\sO_{X^\circ}(K_{X^\circ})\otimes\phi^*_{|X\circ}\sL_{|Y^\circ}\big)$ maps to 
$c_1(\sN_{|X^\circ})^q\in H^q\big(X^\circ,\Omega^q_{X^\circ}\big)$.
Let $\beta \in H^0\big(Y^\circ,R^q{\phi_{|X^\circ}}_*\sO_{X^\circ}({K_{X^\circ}})\otimes \sL_{|Y^\circ}\big)$ denotes the image of
$\alpha^q$ under the edge map
$$H^q\big(X^\circ,\sO_{X^\circ}(K_{X^\circ})\otimes\phi^*_{|X\circ}\sL_{|Y^\circ}\big) \to  H^0\big(Y^\circ,R^q{\phi_{|X^\circ}}_*\sO_{X^\circ}({K_{X^\circ}})\otimes \sL_{|Y^\circ}\big).$$
Now, consider the following commutative diagram:

\centerline{
\xymatrix{
H^q\big(X^\circ,\sO_{X^\circ}(K_{X^\circ})\otimes\phi^*_{|X\circ}\sL_{|Y^\circ}\big) \ar[rr]\ar[d] & &   H^q\big(X^\circ,\Omega^q_{X^\circ}\big) \ar[dd] \\
H^0\big(Y^\circ,R^q{\phi_{|X^\circ}}_*\sO_{X^\circ}({K_{X^\circ}})\otimes \sL_{|Y^\circ}\big)  \ar[d]^{e}\ar@/^2pc/[rrd]^{c} & & \\
R^q{\phi_{|X^\circ}}_*\sO_{X^\circ}({K_{X^\circ}})\otimes\sL_{|Y^\circ}\otimes \mathbb{C}(y)\ar[r] & H^q\Big(\widehat{F},\sO_{\widehat{F}}(\nu^*\big({K_X}_{|F})\big)\otimes \nu^*(\sL_{|F})\Big)\ar[r] & H^q\big(\widehat{F},\Omega^q_{\widehat{F}}\big).
}
}
\noindent Then $$c(\beta) =(-1)^q\cdot \big(\nu^*(c_1(X)_{|F})\big)^q\in H^q\big(\widehat{F},\Omega^q_{\widehat{F}}\big).$$ Since $-K_X$ is $\phi$-ample and $\dim \widehat{F}=q$, we have $\big(\nu^*(c_1(X)_{|F})\big)^q\neq 0$, and hence $e(\beta)\neq 0$.
On the other hand, by \cite[Proposition 7.6]{kollar_higher1}, we have
$R^q{\phi_{|X^\circ}}_*\sO_{X^\circ}({K_{X^\circ}})\cong \sO_{Y^\circ}(K_{Y^\circ})$, and thus $\beta$ is a nowhere vanishing section of $\sO_{Y^\circ}(K_{Y^\circ})\otimes\sL_{|Y^\circ}$. This implies that
$\sL\cong\sO_Y(-K_Y)$
since $\codim\, Y \setminus Y^\circ \ge 2$, proving our claim.

\medskip

Since $\sG$ is projectable under $\phi$, the natural map 
$T_{X/Y}\oplus\sG \to T_X$ is generically injective. A straightforward local computation shows that
$\det(T_{X/Y})\cong \sO_X\big(-K_{X/Y}+R(\phi)\big)$, 
where $R(\phi)$ denotes the ramification divisor of $\phi$, and thus
$$\det(T_{X/Y}\oplus\sG)\cong \det(T_{X/Y})\otimes \det(\sG)\cong \sO_X\big(-K_{X}+R(\phi)\big).$$
This implies that $R(\phi)=0$, and that 
$T_{X/Y}\oplus\sG \cong T_X.$
It follows that $T_{X/Y}$ and that $\sG$ are locally free sheaves on $X$ and that $Z(\sG)=\emptyset$. From \cite[Lemma 4.1]{druel_fol_fano}, we conclude that $\phi$ is a smooth morphism.
Moreover, $\sG$ induces a flat (holomorphic) connection on $\phi$, completing the proof of the proposition.
\hfill $\Box$
\end{proof}

\section[Codimension 2 (regular) foliations with numerically trivial canonical class]{Codimension 2 (regular) foliations with numerically trivial \\ canonical class}

In this section we prove Theorem \ref{thmintro:cd2_non_nef}. 
Notice that Theorem \ref{thmintro:cd2_non_nef} is an immediate consequence of
Theorem \ref{thm:main2} below. 

The statement of Proposition \ref{prop:singular_set_cd1} is contained in
\cite[Corollary 4.7]{loray_pereira_touzetv3}. We first give a new proof of this result.

\begin{prop}\label{prop:singular_set_cd1}
Let $X$ be a complex projective manifold, and let $\sG \subset T_X$ be a foliation of codimension $1$ on $X$ with $K_\sG\equiv 0$. Suppose that $\dim Z(\sG) \le \dim X - 3$. 
Suppose furthermore that $K_X$ is not nef. Then $Z(\sG)=\emptyset$.
\end{prop}

\begin{proof}
Let $\phi\colon X \to Y$ be the contraction of a $K_X$-negative extremal ray $R \subset \NE(X)$.
By Lemma \ref{lemma:dimension_fiber_contraction},
$\dim Y = \dim X - 1$ and $\phi$ is equidimensional. Applying \cite[Theorem 1.2]{wisn_crelle}, we see that $\phi$ is a conic bundle. Applying Corollary \ref{cor:projectable_criterion}, we see that $\sG$ is projectable under $\phi$.
The proposition then follows from Proposition \ref{prop:foliation_equidim_fibration}.
\hfill $\Box$
\end{proof}

\begin{thm}\label{thm:main2}
Let $X$ be a complex projective manifold, and let
$\sG \subset T_X$ be a foliation of codimension $2$ on $X$ with $K_\sG\equiv 0$.
Suppose that $\dim Z(\sG) \le \dim X - 4$. 
Suppose furthermore that $K_X$ is not nef.
Then $Z(\sG)=\emptyset$ and one of the following holds.
\begin{enumerate}[\rm (1)]
\item There exist a complex projective manifold $Y$, a regular codimension 2 foliation $\sG_Y$ on $Y$ with 
$K_{\sG_Y}\equiv 0$, and a compact leaf $Z \subset Y$ of $\sG_Y$ such that $X$ 
is the blow-up of $Y$ along $Z$. Moreover, $\sG$ is the pull-back of $\sG_Y$.
\item There exists a conic bundle structure $\phi\colon X \to Y$ on $X$ onto a complex projective manifold $Y$, 
$\sG$ is generically transverse to $\phi$ and induces a 
regular codimension 1 foliation $\sG_Y$ on $Y$ with 
$K_{\sG_Y}\equiv 0$. Moreover, 
the discriminant locus of $\phi$ is either empty or a union of leaves of $\sG_Y$.
\item There exists a smooth morphism $\phi\colon X \to Y$ onto a complex projective manifold $Y$ 
of dimension $\dim Y = \dim X -2$
with 
$K_Y\equiv 0$, and $\sG$ yields a flat holomorphic connection on $\phi$. 
\end{enumerate}
\end{thm}

\begin{proof}
Let $\phi\colon X \to Y$ be the contraction of a $K_X$-negative extremal ray $R \subset \NE(X)$, and let
$F$ be an irreducible component of some fiber of $\phi$. Applying 
Lemma \ref{lemma:dimension_fiber_contraction} to $\phi$, we see that one of the following holds:

\begin{enumerate}[\rm (1)]
\item $\dim Y = \dim X - 2$ and $\phi$ is equidimensional,
\item $\dim Y = \dim X - 1$ and $\phi$ is equidimensional,
\item $\dim Y = \dim X$ and $\dim F \le 1$. 
\end{enumerate}

By Corollary \ref{cor:projectable_criterion}, $\sG$ is projectable under $\phi$.

If we are in case (1), then the theorem follows from Proposition \ref{prop:foliation_equidim_fibration}.

\medskip

Suppose that we are in case (2). Applying \cite[Theorem 1.2]{wisn_crelle}, we see that $\phi$ is a conic bundle.
By Proposition \ref{prop:foliation_conic_bundle}, we have $K_{\sG_Y}\equiv 0$
and $\dim Z(\sG_Y) \le \dim Y- 3$. 
By \cite[Theorem 5.4]{loray_pereira_touzetv4} and Proposition \ref{prop:singular_set_cd1} applied to $\sG_Y$, 
the singular set $Z(\sG_Y)$ is empty.
This implies that $Z(\sG)=\emptyset$ as well (see \ref{say:morphisms_and_foliations}).
Apply again Proposition \ref{prop:foliation_conic_bundle} to 
conclude that the discriminant locus $\Delta$ of $\phi$ is either empty or a union of leaves of $\sG_Y$. 

\medskip

Finally, suppose that we are in case (3). 
By \cite[Theorem 1.2]{wisn_crelle}, $Y$ is smooth and $\phi$ is the blow-up of $Y$ along a codimension 2 submanifold $Z \subset Y$.
By Proposition \ref{prop:foliation_blow_up}, $K_{\sG_Y}\equiv 0$,  
$\dim Z(\sG_Y) \le \dim Z(\sG) \le \dim Y-4$, and $Z(\sG)$ is empty if and only if so is $Z(\sG_Y)$.

Suppose that $Z(\sG)$ is non-empty. In particular, $Z(\sG_Y)$ is also non-empty. If moreover $K_Y$ is not nef, then there
exists an elementary Fano-Mori contraction on $Y$ as in case (3).
This yields a finite sequence of contractions 
$$Y_0:=X\to Y_1:=Y \to Y_2 \to \cdots \to Y_i \to Y_{i+1} \to \cdots\to Y_m$$ 
where $Y_i$ is a complex projective manifold, 
$\phi_i\colon Y_i \to Y_{i+1}$ is the blow-up of a codimension 2 submanifold $Z_{i+1}\subset Y_{i+1}$, and $K_{Y_m}$ is nef.
Note that the process ends since $\rho(Y_{i+1})<\rho(Y_i)$.
Moreover, the projection $\sG_i$ of $\sG$ on $Y_i$ is a codimension $2$ foliation 
with $K_{\sG_i}\equiv 0$ and non-empty singular set $Z(\sG_i)$. 
This contradicts \cite[Theorem 5.4]{loray_pereira_touzetv4} applied to $\sG_m$, proving
that $Z(\sG)=\emptyset$. Applying Proposition \ref{prop:foliation_blow_up} again, we see that
$Z(\sG_Y)=\emptyset$, and that $Z$ is a leaf of $\sG_Y$.
This completes the proof of the theorem.
\hfill $\Box$
\end{proof}

\section[Algebraically integrable regular foliations with numerically trivial canonical class]{Algebraically integrable regular foliations with numerically \\ trivial canonical class}

In this section we prove Theorem \ref{thmintro:cd2_psef}.
The proof relies on a global version of Reeb stability theorem (see Proposition \ref{prop:alg_int_zer_can_class}), which we establish first.

The following notation is used in the formulation of Proposition \ref{prop:alg_int_zer_can_class}.

\begin{defn} \rm
Let $X$ be a normal projective variety, let $H$ be an ample Cartier divisor on $X$, and let $\sG$ be a reflexive coherent sheaf of $\sO_X$-modules. We say that $\sG$ is \textit{strongly stable with respect to $H$} if, for any 
normal projective variety $Z$ and any
generically finite surjective morphism $f \colon Z \to X$, the reflexive pull-back sheaf $(f^*\sG)^{**}$ is $f^{*}H$-stable.
\end{defn}

\begin{say}[The holonomy group of a stable reflexive sheaf] \rm
Let $X$ be a normal complex projective variety, and let $\sG$ be a reflexive sheaf on $X$. 
Suppose that $\sG$ is stable with respect to an ample Cartier divisor $H$ and that $\mu_H(\sG)=0$.
For a sufficiently large positive integer $m$,
let $C \subset X$ be a general complete intersection curve of elements in $|mH|$. Let $x \in C$.
By the restriction theorem of Mehta and Ramanathan,
the locally free sheaf $\sG_{|C}$ is stable with $\deg(\sG_{|C})=0$, and hence
it corresponds to a unique unitary representation $\rho\colon\pi_1(C,x)\to\mathbb{U}(\sG_x)$ 
by a result of Narasimhan and Seshadri (\cite{narasimhan_seshadri65}). The \emph{holonomy group} $\Hol_x(\sG)$
of $\sG$ is the Zariski closure of $\rho\big(\pi_1(C,x)\big)$ in $\textup{GL}(\sG_x)$. It does not depend on $C \ni x$ provided that $m$ is large enough (see \cite{balaji_kollar}).
\end{say}

We will need the following observation.

\begin{lemma}\label{lemma:holonomy_group_versus_strong_stability}
Let $X$ be a normal complex projective variety, let $x$ be a general point on $X$, and let $\sG$ be a coherent sheaf of $\sO_X$-modules.
Suppose that $\sG$ is stable with respect to an ample divisor $H$ and that $\mu_H(\sG)=0$.
Suppose furthermore that its holonomy group $\Hol_x(\sG)$ is connected. 
Then $\sG$ is strongly stable with respect to $H$.
\end{lemma}

\begin{proof}
Let $Z$ be a normal projective variety, and let
$g \colon Z \to X$ be a generically finite surjective morphism. The map $g$ factorizes into 
$Z \to Y \to X$ where $Y$ is a normal projective variety, $b\colon Z \to Y $ is a birational map, and $f\colon Y \to X$
is a finite cover.
The same argument used in the proof of \cite[Lemma 6.22]{bobo} shows that the reflexive pull-back sheaf $(f^*\sG)^{**}$ is $f^{*}H$-stable. One only needs to replace the use of \cite[Theorem 1]{kempf}
with \cite[Lemma 3.2.3]{HuyLehn}.

We argue by contradiction and assume that $(g^*\sG)^{**}$ is not $g^{*}H$-stable. It follows that there exists 
$\sE \subset (g^*\sG)^{**}$ with $0 < \textup{rank}\, \sE < \textup{rank}\, (g^*\sG)^{**}$
and $\mu_{g^{*}H}(\sE) \ge \mu_{g^{*}H}\big((g^*\sG)^{**}\big)=\mu_{H}(\sG)=0$. Then 
$b_*\sE \subset (f^*\sG)^{**}$ since $b_*\big((g^*\sG)^{**}\big)$ and $(f^*\sG)^{**}$ agree
over an open subset with complement of codimension at least two. On the other hand, 
$\mu_{f^*H}(b_*\sE)=\mu_{g^*H}(\sE)=0$. This  yields a contradiction, finishing the proof of the lemma.
\hfill $\Box$
\end{proof}

The following result is probably well-known to experts. We include a proof here for the reader's convenience.

\begin{lemma}\label{lemma:alg_int_zer_can_class}
Let $X$ be a complex projective manifold, and $\phi\colon X \to Y$ be a smooth morphism onto a complex projective manifold. Suppose that $K_X$ is pseudo-effective, and suppose that $K_{X/Y}\equiv 0$. Then there exist complex projective manifolds $B$ and $F$ as well as a 
finite \'etale cover $f\colon B \times F \to X$ such that 
$T_{B\times F/B}=f^*T_{X/Y}$. 
\end{lemma}

\begin{proof}
By \cite[Theorem 5.2]{loray_pereira_touzetv4}, the divisor $K_{X/Y}$ is a torsion point, and hence, by replacing $X$ with a finite \'etale cover, if necessary, we may assume without loss of generality that $K_{X/Y}\sim 0$. 
Applying \cite[Theorem 5.2]{loray_pereira_touzetv4} again, we see that $\phi$ admits a holomorphic connection. In particular, $\phi$ is a locally trivial fibration for the Euclidean topology. Let $F$ be any fiber of $\phi$. 
By the adjunction formula, $K_F\sim0$.
Let $y\in Y$, and denote by $X_y\cong F$ the fiber $\phi^{-1}(y)$.
Let $\textup{Aut}^\circ(X_y)$ denotes the neutral component of 
the automorphism group $\textup{Aut}(X_y)$ of $X_y$. Then $\textup{Aut}^\circ(X_y)$ is an abelian variety
of dimension $h^0(F,T_F)$.
Recall from \cite[Expos\'e VI$_{\textup{B}}$, Th\'eor\`eme 3.10]{sga3} 
that the algebraic groups
$\textup{Aut}^\circ(X_y)$ fit together to form an abelian scheme $\sA$
over $Y$. Since $\sA$ is locally trivial, there exist 
an abelian variety $A$, and a finite \'etale
cover $Y_1 \to Y$ such that $\sA \times_Y Y_1 \cong A \times Y_1$
as group schemes over $Y_1$. 
This follows from the fact that there is a fine moduli scheme for polarized abelian varieties of dimension $g$, with level $N$ structure and polarization of degree $d$ provided that $N$ is large enough.
In particular, $A$ acts faithfully on $X_1:=X\times_Y Y_1$. 
By \cite[Proof of Theorem 1.2, page 10]{brion_action}, there exist a finite \'etale cover 
$X_2$
of $X_1$ 
equipped with a faithful action of $A$, and an $A$-isomorphism
$X_2 \cong A \times Z_2$ for some projective manifold $Z_2$, where $A$ acts 
trivially on $Z_2$ and
diagonally on $A \times Z_2$. 
One readily checks that there exists a smooth morphism with connected fibers $\phi_2\colon Z_2 \to Y_2$ 
as well as a finite \'etale cover $Y_2 \to Y_1$ and a commutative diagram:

\centerline{
\xymatrix{
A \times Z_2 \ar[r]\ar[d] & X_1 \ar[dd] \\
Z_2\ar[d] & \\
 Y_2 \ar[r] & Y_1.\\
}
}
\noindent This implies that $K_{A \times Z_2/Y_2}\sim 0$, and thus we also have 
$K_{Z_2/Y_2}\sim 0$.
Repeating the process finitely many times, if necessary, we may assume without loss of generality that 
any fiber $F_2$ of $\phi_2$ satisfies $h^0(F_2,T_{F_2})=0$.
Notice that $K_{F_2}\sim 0$ by the adjunction formula. Moreover, 
$$
\begin{array}{ccll}
h^1(F_2,\sO_{F_2}) & =  & h^{\dim F_2 -1}(F_2,\sO_{F_2})
& \text{ by Serre duality}\\
& = & h^0(F_2,\Omega_{F_2}^{\dim F_2 -1})   
& \text{ by Hodge symmetry} \\
& = & h^0(F_2,T_{F_2})
& \text{ for $\Omega_{F_2}^{\dim F_2 -1}\cong T_{F_2}$}\\
& = & 0.
\end{array}
$$
Let $H_2$ be an ample divisor on $Z_2$. Since $\phi_2$ admits a flat holomorphic connection, it is given by 
a representation
$$\pi_1(Y_2,y) \to \textup{Aut}(F_2,{H_2}_{|F_2})$$
where $y:=\phi_2 (F_2)$, and where $\textup{Aut}(F_2,{H_2}_{|F_2})$ denotes the group
$\{u \in \textup{Aut}(F_2)\,|\,u^*{H_2}_{|F_2}\equiv {H_2}_{|F_2}\}$.
Since $h^1(F_2,\sO_{F_2})=0$, the Picard group of $F_2$ is discrete. The subgroup 
$\textup{Pic}^{\tau}(F_2)\subset \textup{Pic}(F_2)$ of invertible sheaves with numerically trivial first Chern class is therefore finite. By replacing $H_2$ with $mH_2$ for some positive integer $m$, if necessary, we may assume without loss of generality, that  
$\textup{Aut}(F_2,{H_2}_{|F_2})=\{u \in \textup{Aut}(F_2)\,|\,u^*{H_2}_{|F_2}\sim {H_2}_{|F_2}\}$.
This implies that $\textup{Aut}(F_2,{H_2}_{|F_2})$ is an affine algebraic group. It follows that 
$\textup{Aut}(F_2,{H_2}_{|F_2})$
is finite since 
since $h^0(F_2,T_{F_2})=0$.
Therefore, replacing $Y_2$ with a finite \'etale cover, if necessary, we may assume that 
$Z_2 \cong Y_2 \times F_2$ as varieties over $Y_2$. This finishes the proof of the lemma.
\hfill $\Box$
\end{proof}

The same argument used in the proof of Lemma \ref{lemma:alg_int_zer_can_class} shows that the following holds.

\begin{lemma}\label{lemma:alg_int_zer_can_class2}
Let $X$ be a complex projective manifold, and $\phi\colon X \to Y$ be a smooth morphism onto a complex projective manifold. Suppose that $K_X$ is pseudo-effective, and suppose that $K_{X/Y}\sim 0$. 
Suppose furthermore that the irregularity of any fiber of $\phi$ is zero. Then there exist complex projective manifolds $Y_1$ and $F$ as well as a finite \'etale cover $Y_1 \to Y$ such that $Y_1 \times_Y X \cong Y_1 \times F$ as varieties over $Y_1$.
\end{lemma}

The following is a global version of Reeb stability theorem for regular foliations with numerically trivial canonical class on complex projective manifolds with pseudo-effective canonical divisor. See \cite[Theorem 3.2]{hwang_viehweg} for a somewhat related result.

\begin{prop}\label{prop:alg_int_zer_can_class}
Let $X$ be a complex projective manifold, let $H$ be an ample divisor on $X$, and let $\sG\subset T_X$ be a regular foliation with $K_\sG\equiv 0$. Suppose that $K_X$ is pseudo-effective.
Suppose furthermore that $\sG$ is algebraically integrable, and
that one of the following two conditions holds.
\begin{enumerate}[\rm (1)]
\item The sheaf $\sG$ is $H$-semistable and $c_2(\sG)\equiv 0$.
\item The sheaf $\sG$ is $H$-strongly stable. 
\end{enumerate}
Then there exist complex projective manifolds $F$ and $Y$, and a finite \'etale cover $f \colon Y \times F \to X$ such that $f^{-1}\sG$ is induced by the projection morphism $Y \times F \to Y$.
\end{prop}

\begin{proof}
By \cite[Theorem 5.2]{loray_pereira_touzetv4}, the divisor $K_{\sG}$ is a torsion point, and hence, by replacing $X$ with a finite \'etale cover, if necessary, we may assume without loss of generality that $K_{\sG}\sim 0$. 

\medskip

Suppose first that $\sG$ is $H$-semistable, and that $c_2(\sG)\equiv 0$. Applying \cite[Theorem 5.2]{loray_pereira_touzetv4}, we see that there exists a regular foliation $\sE$ on $X$ such that $T_X\cong \sG\oplus \sE$. By \cite[Lemma 5.8]{bobo}, there exists a finite \'etale cover $f_1\colon X_1 \to X$ such that $f_1^*\sG \cong \sO_{X_1}^{\oplus r}$, where $r:=\textup{rank} \,\sG$. 
We claim that the neutral component 
$\textup{Aut}(X_1)^\circ$ of $\textup{Aut}(X_1)$ is an abelian variety. 
Suppose otherwise.
Then, by Chevalley's structure theorem, 
$\textup{Aut}(X_1)^\circ$ contains a positive dimensional affine subgroup. Hence, it contains an algebraic subgroup isomorphic to $\mathbb{G}_m$ or $\mathbb{G}_a$. This implies that $X$ is uniruled, yielding a contradiction since $K_X$ is pseudo-effective. Set $A:=\textup{Aut}(X_1)^\circ$. By \cite[Proof of Theorem 1.2, page 10]{brion_action}, replacing $X_1$ 
with a further \'etale cover, if necessary, we may assume that $X_1 \cong A \times Y$ and that 
$h^0(Y,T_Y)=0$. In particular, we must have $f_1^*\sG \subset T_{A \times Y/Y}$. 
Thus, there exists a linear foliation 
$\sH$ on $A$ with algebraic leaves such that
$f_1^{-1}\sG = p^{-1}\sH \cap T_{A \times Y/Y}$, where $p\colon X_1 \to A$ denotes the projection.
Let $F \subset A$ be a leaf of $\sH$, and let $B \subset A=\textup{Aut}^\circ(A)$ be the connected component of the group 
of elements $a \in A$ such that
$a+F = F$; $B$ is an algebraic group and its Lie algebra is the kernel of the natural map
$$\textup{Lie}(A) \cong H^0(A,T_A) \to H^0\big(F,{T_A}_{|F}\big) \to H^0(F,\sN_{F/A}).$$
One readily checks that $\sN_{F/A}\cong \sO_F^{\oplus \dim A - \dim F}$
and that the above map is surjective. 
We conclude that $B$ is an abelian subvariety of $A$ of dimension $r$. Note that $a+B$ is a leaf of $\sH$ for all $a \in A$ since $\sH$ is linear.
By Poincar\'e's complete reducibility theorem, there exists an abelian subvariety
$C \subset A$ such that the natural morphism $B \times C \to A$ is an isogeny (see \cite[Theorem IV.1]{mumford_av}).
If $f \colon B \times C \times Y \to X$ denotes the induced morphism, then $f^{-1}\sG$ is given by the projection
$B \times C \times Y \to C \times Y$, proving
Proposition \ref{prop:alg_int_zer_can_class} in this case.

\medskip

Suppose from now on that $\sG$ is $H$-strongly stable. We may also assume that $c_2(\sG)\not\equiv 0$.

By the holomorphic version of Reeb stability theorem (see \cite[Proposition 2.5]{hwang_viehweg}), $\sG$ is induced by a morphism $\phi\colon X \to Y$ onto a normal projective variety.
Applying \cite[Theorem 2.7]{hwang_viehweg}, we see that there exists a finite set $\Gamma$ of indices, \'etale morphisms $f_\gamma\colon X_\gamma \to X$ for all $\gamma\in \Gamma$ as well as smooth projective morphisms with connected fibers $\phi_\gamma\colon X_\gamma \to Y_\gamma$ such that $f_\gamma^{-1}\sG$ is induced by $\phi_\gamma$ and such that 
$X = \cup_{\gamma\in \Gamma}f_\gamma(X_\gamma)$. The following properties hold in addition.
There is a commutative diagram:

\centerline{
\xymatrix{
X_\gamma \ar[r]^{f_\gamma}\ar[d]_{\phi_\gamma} & X \ar[d]^{\phi} \\
 Y_\gamma \ar[r]_{g_\gamma} & Y,\\
}
}
\noindent $g_\gamma(Y_\gamma)$ is a Zariski open set, and $g_\gamma\colon Y_\gamma \to g_\gamma(Y_\gamma)$ is a finite morphism
(see Proof of \cite[Lemma 2.6]{hwang_viehweg}). In particular, $X_\gamma$ identifies with the normalization of $Y_\gamma \times_Y X$.
Let $g_1\colon Y_1 \to Y$ be the normalization of $Y$ in the Galois closure of the compositum of the function fields
$\mathbb{C}(Y_\gamma)$. 
Let $X_1$ be the normalization of $Y_1\times_Y X$, and denote by $f_1 \colon X_1 \to X$ 
and $\phi_1\colon X_1 \to Y_1$
the natural morphisms. 
Notice that $\phi_1$ is smooth with connected fibers since $g_1 \colon Y\to Y$ factors through 
$g_\gamma\colon Y_\gamma \to Y$ over $g_\gamma(Y_\gamma)$ and $Y=\cup_{\gamma\in \Gamma}g_\gamma(Y_\gamma)$. 
Finally, let $g_2\colon Y_2 \to Y_1$ be a desingularization of $Y_1$, and set $X_2 := Y_2\times_{Y_1} X_1$. Observe that $X_2$ is a smooth projective variety, and denote by $f_2\colon X_2 \to X_1$ and $\phi_2 \colon X_2 \to Y_2$ the natural morphisms. There is a commutative diagram:

\centerline{
\xymatrix{
X_2\ar[rr]^{f_2,\textup{ birational}}\ar[d]_{\phi_2,\textup{ smooth}} && X_1 \ar[rr]^{f_1,\textup{ finite}}\ar[d]_{\phi_1,\textup{ smooth}} && X \ar[d]^{\phi} \\
Y_2\ar[rr]_{g_2,\textup{ birational}} && Y_1 \ar[rr]_{g_1,\textup{ finite}} && Y.\\
}
}
Let $R(\phi)$ be the ramification divisor of $\phi$. By Example \ref{exmp:can_div_fol_morphism} and Lemma \ref{lemma:pull_back_fol_and_finite_cover}, we have 
$$K_{X_1/Y_1}\sim f_1^*\big(K_{X/Y}-R(\phi)\big)\sim f_1^*K_\sG\sim 0.$$ This immediately implies that 
$$K_{X_2/Y_2}\sim f_2^*K_{X_1/Y_1}\sim 0.$$ 
Applying Lemma \ref{lemma:alg_int_zer_can_class} to $\phi_2$, we see that
there exist complex projective manifolds $B$ and $S$ as well as a 
finite \'etale cover $$h\colon B \times S \to X_2$$ such that  
$$h^*T_{X_2/Y_2}\cong T_{B\times S/B}.$$ It follows that
$K_{S}\sim 0$. By replacing $S$ with a further \'etale cover, if necessary, we may assume that $S$ decomposes into the product of an abelian variety and a simply-connected complex projective manifold.
In particular, $T_S$ is polystable with respect to any polarization. 
We claim that $S$ is simply-connected.
By assumption, the sheaf $\big((f_1\circ f_2\circ h)^*\sG\big)^{**}$ is 
stable with respect to $(f_1\circ f_2\circ h)^*H$. 
Let $C \subset B \times S$ be a smooth complete intersection curve of elements of $|(f_1\circ f_2\circ h)^*mH|$ for some sufficiently large integer $m$. 
Applying \cite[Theorem 5.2]{langer_ss_sheaves}, we see that ${\big((f_1\circ f_2\circ h)^*\sG\big)^{**}}_{|C}$ is stable. On the other hand, by Lemma \ref{lemma:pull_back_fol_and_finite_cover}, 
we have ${\big((f_1\circ f_2\circ h)^*\sG\big)^{**}}_{|C}\cong {T_{B\times S/B}}_{|C}$.
This implies that $T_S$ is actually stable (with respect to any polarization). 
If $S$ is a genus one curve, then $\sG$ is a line bundle, and hence $c_2(\sG)\equiv 0$, yielding a contradiction. This shows that $S$ is a simply-connected complex projective manifold, proving our claim.

The irregularity of any fiber $F$ of $\phi_2$ is zero because $S$ is simply-connected.
Applying Lemma \ref{lemma:alg_int_zer_can_class2}, we see that there exists a finite \'etale cover 
$g_3 \colon Y_3 \to Y_2$ such that 
$$X_3:=Y_3\times_{Y_2} X_2 \cong Y_3 \times F.$$
\noindent We obtain the following commutative diagram:

\centerline{
\xymatrix{
X_3\cong Y_3 \times F\ar[rr]^{f_3,\textup{ \'etale}}\ar[d]_{\textup{projection}} && X_2\ar[rr]^{f_2,\textup{ birational}}\ar[d]_{\phi_2,\textup{ smooth}} && X_1 \ar[rr]^{f_1,\textup{ finite}}\ar[d]_{\phi_1,\textup{ smooth}} && X \ar[d]^{\phi} \\
Y_3\ar[rr]_{g_3,\textup{ \'etale}} && Y_2\ar[rr]_{g_2,\textup{ birational}} && Y_1 \ar[rr]_{g_1,\textup{ finite}} && Y.\\
}
}
\noindent Let $Y_4$ be the normalization of $Y_1$ in the function field $\mathbb{C}(Y_3)$, and set 
$X_4 = Y_4\times_{Y_1} X_1$. 
This gives another commutative diagram:

\centerline{
\xymatrix{
X_3\cong Y_3 \times F \ar[rr]_{\textup{birational}}\ar[d]_{\textup{projection}}\ar@/^1.5pc/[rrrr]^{f_2\circ f_3} && X_4 \ar[d]\ar[rr]_{\textup{finite}} && X_1 \ar[d]\\
 Y_3 \ar[rr]^{\textup{birational}}\ar@/_1.5pc/[rrrr]_{g_2\circ g_3} && Y_4 \ar[rr]^{\textup{finite}} && Y_1.\\
}
}

\noindent Note that $X_3$ identifies with $Y_3 \times_{Y_4} X_4 \cong X_3$. 
Using the rigidity lemma, we see that the second projection $X_3 \cong Y_3 \times F \to F$ factors into 
$X_3 \to X_4 \to F$. This immediately implies that 
$X_4 \cong Y_4 \times F$ over $Y_4$.

By replacing $Y_4$ with a further finite cover, if necessary, we may assume without loss of generality that
the finite cover $Y_4 \to Y$ is Galois. In particular,
there is a finite group $G$ acting on $Y_4$ with quotient $Y$. 

The group $G$ also acts on $Y_4 \times F$ since $Y_4 \times F$ identifies with the normalization of
$Y_4\times_{Y} X$. Since $h^0(F,T_{F})=0$, $G$ acts on $F$ and its action on $Y_4\times F$ is the diagonal action. Let $G_1$ denote the kernel of the induced morphism of groups $G \to \textup{Aut}(F)$.
By replacing $Y_4$ by $Y_4/G_1$, $X_4$ by $X_4/G_1 \cong (Y_4/G_1)\times F$, and $G$ by $G/G_1$, if necessary, we may assume without loss of generality that $G \subset \textup{Aut}(F)$. Then
the quotient map $Y_4 \times F\to (Y_4 \times F)/G\cong X$ is automatically
\'etale in codimension one, and hence \'etale by the Nagata-Zariski purity theorem.
This finishes the proof of Proposition \ref{prop:alg_int_zer_can_class}.
\hfill $\Box$
\end{proof}

The following conjecture is due to Pereira and Touzet (see \cite{pereira_touzet}).

\begin{conj}[Pereira--Touzet]\label{conj:pereira_touzet}
Let $\sG$ be a regular foliation on a complex projective manifold $X$ with $K_X$ pseudo-effective. Suppose that 
$c_1(\sG)\equiv 0$ and that $c_2(\sG)\not\equiv 0$. If $\sG$ is stable with respect to some polarization, then
it is algebraically integrable.
\end{conj}

Building on our recent algebraicity criterion for leaves of algebraic foliations \cite[Theorem 1.6]{bobo}, we confirm Conjecture \ref{conj:pereira_touzet} in some special cases.

\begin{prop}\label{prop:conj_holds_rank_3}
Conjecture \ref{conj:pereira_touzet} holds for foliations of rank at most $3$.
\end{prop}

\begin{proof}
We maintain notation and assumptions of Conjecture \ref{conj:pereira_touzet}.
Suppose furthermore that $\sG$ has rank $r \le 3$.

Suppose that there exists a finite \'etale cover $g\colon Y \to X$ such that $g^*\sG$ is not stable.
Applying \cite[Lemma 2.1]{pereira_touzet}, we see that 
there exist non-zero vector bundles $(\sG_i)_{i \in I}$, stable with respect to any polarization with 
$c_1(\sG_i)\equiv 0$ such that 
$f^{*}\sG\cong\oplus_{i \in I}\sG_i$.
By \cite[Proposition 8.1]{bobo}, to prove that $\sG$ is algebraically integrable, it sufficies to prove that 
for some $i \in I$, $\sG_i$ has algebraic leaves.
Therefore, by replacing $X$ with a finite \'etale cover, if necessary, we may assume that 
for any finite \'etale cover $g\colon Y \to X$, $g^*\sG$ is stable with respect to any polarization.

By replacing $X$ with a further finite \'etale cover, if necessary, we may also assume that $K_\sG\sim 0$ by \cite[Theorem 5.2]{loray_pereira_touzetv4}. Let $H$ be an ample divisor on $X$ such that $c_2(\sG)\cdot H^{n-2}\neq 0$. By \cite[Theorem 5.2]{loray_pereira_touzetv4}, there exists a regular foliation $\sE$ on $X$ such that $T_X\cong \sG\oplus \sE$. 
The proposition then follows from \cite[Theorem 1.6]{bobo}. 
\hfill $\Box$
\end{proof}

We are now in position to prove our main result. Note that Theorem \ref{thmintro:cd2_psef} is an immediate consequence of Proposition \ref{prop:conj_holds_rank_3} and of Theorem \ref{thm:main3} below.

\begin{thm}\label{thm:main3}
Let $X$ be a complex projective manifold, and let
$\sG \subset T_X$ be a regular foliation of rank $r$ with $K_\sG\equiv 0$.
Suppose that $K_X$ is pseudo-effective. Suppose furthermore that Conjecture \ref{conj:pereira_touzet} holds for foliations of rank at most $r$.
Then there exist complex projective manifolds $Y$ and $F$, a finite \'etale cover $f \colon Y \times F \to X$, and a regular foliation $\sH$ on $Y$ with $c_1(\sH)\equiv 0$ and $c_2(\sH)\equiv 0$
such that $f^{-1}\sG$ is the pull-back of $\sH$.
\end{thm}

\begin{proof}
Recall from \cite[Lemma 2.1]{pereira_touzet}, that $\sG$ is polystable with respect to any polarization.
Thus, there is a decomposition $\sG\cong\oplus_{i\in I}\sG_i$ of $\sG$ into involutive sub-vector bundles of $T_X$ such that $\sG_i$ is stable with respect to any polarization. Notice that $c_1(\sG_i)\equiv 0$ for all indices $i\in I$. 
Notice also that it suffices to prove the theorem for the foliations $\sG_i$.
Therefore, we may assume without loss of generality that $\sG$ is 
stable with respect to any polarization.

Suppose that there exists a finite \'etale cover $g\colon Y \to X$ such that $g^*\sG$ is not stable.
Applying \cite[Lemma 2.1]{pereira_touzet}, we see that 
there exist non-zero vector bundles $(\sG_j)_{j \in J}$, stable with respect to any polarization with 
$c_1(\sG_j)\equiv 0$ such that 
$f^{*}\sG\cong\oplus_{j \in J}\sG_j$. 
As before, it suffices to prove the theorem for $\sG_j$ for all indices $j\in J$.
Therefore, by replacing $X$ with a further finite \'etale cover, if necessary, we may assume that 
for any finite \'etale cover $g\colon Y \to X$, $g^*\sG$ is stable with respect to any polarization.

Let $H$ be an ample divisor on $X$. By \cite[Lemma 40]{balaji_kollar} (see also \cite[Lemma 6.20]{bobo}), 
there exists a finite \'etale cover $f\colon \wt X \to X$ such that $\Hol_{\wt x}(f^{*}\sG)$ is connected, where 
$\wt x$ is a point on $\wt X$.
Applying Lemma \ref{lemma:holonomy_group_versus_strong_stability}, we see that $f^*\sG$ is strongly stable with respect to $f^*H$. If $c_2(f^*\sG)\equiv 0$, then the statement is obvious. If 
$c_2(f^*\sG)\not\equiv 0$, the theorem then follows from Proposition \ref{prop:alg_int_zer_can_class}.
\hfill $\Box$
\end{proof}

\section{Poisson manifolds -- Generalized Bondal conjecture}

In this section, we address the generalized Bondal conjecture. We first recall the basic facts concerning Poisson manifolds.

\begin{say}[Poisson structures] \rm
A (holomorphic) \textit{Poisson structure} on a complex manifold $X$ is a bivector field 
$\tau \in H^0(X,\wedge^2T_X)$, such that the bracket $\{f,g\}:=\langle \tau ,df\wedge dg\rangle$ defines a Lie algebra structure on $\sO_X$. A Poisson structure defines a skew-symmetric  map $\tau ^\sharp:\Omega ^1_X\rightarrow T_X$; the \textit{rank} of $\tau $ at a point $x\in X$ is the rank of $\tau ^\sharp(x)$. It is  even because $\tau ^\sharp$ is skew-symmetric. The data of a Poisson structure of rank $\dim X$ is equivalent to that of a (holomorphic) symplectic structure. In general, we have a partition
$$X=\bigsqcup_{s\textup{ even}}X_s \quad\textup{where}\quad X_s :=\{x\in X\ |\ \textup{rk }\tau^\sharp (x)=s\}.$$
Let $r$ be the generic rank of $\tau$. 
The \textit{generic corank} of $\tau$ is defined as $q:=\dim X -r$.
\end{say}

The following conjecture is due to Beauville (\cite[Conjecture 5]{beauville_problem_list}).

\begin{conj}
Let $(X,\tau)$ be a projective Poisson manifold, and
let $0 \le s <r$ be an even integer.
If $X_{\leq s}:=\bigsqcup_{k\leq s}X_k$
is non-empty, it  contains a component of dimension $>s$. 
\end{conj}

The generalized Bondal conjecture implies that if the degeneracy locus $X \setminus X_r$ of $\tau$ is non-empty, then it contains a component of dimension $>r-2$. This is known to be true if $c_1(X)^{q+1}\neq 0$ in $H^{q+1}(X,\Omega^{q+1}_X)$ (see \cite[Corollary 9.2]{polishchuk}, and also \cite[Proposition 4]{beauville_problem_list}).

\begin{say}[Foliation associated to a Poisson structure]\label{foliation_symplectic} \rm
Let $(X,\tau)$ be a Poisson manifold, and let $r$ be the generic rank of $\tau$. 
The distribution on $X_r$ given by Hamiltonian vector fields
$\tau^\sharp(df)$ is involutive. We denote by $\sG \subset T_X$ the corresponding (possibly singular) foliation on $X$. 
The restriction of $\tau^\sharp$ to $X_r$ induces a non-degenerate
skew-symmetric map $(\sG_{|X_r})^* \to \Omega^1_{X_r}\to \sG_{|X_r}$. 
This implies that
$\det(\sG_{|X_r})\cong \sO_{X_r}$. Set $\sN:=T_X/\sG$. 
Notice that $Z(\sG) \subset X\setminus X_r$. Thus,
if the degeneracy locus $X \setminus X_r$ of $\tau$ has codimension
at least $2$ in $X$, then we have
$$\det(\sG)\cong \sO_X\quad\textup{and}\quad\det(\sN)\cong \sO_X(-K_X).$$
\end{say}

The proof of Theorem \ref{thmintro:generalized_bondal_conjecture} makes use of the following observation.

\begin{lemma}\label{lemma:singular_set_associated_foliation}
Let $(X,\tau)$ be a Poisson manifold, and let $\sG$ be the corresponding foliation.
Suppose that the degeneracy locus $X \setminus X_r$ of $\tau$ has codimension
at least $2$ in $X$. If $Z(\sG)$ is empty, then $X = X_r$.
\end{lemma}

\begin{proof}
Suppose that $Z(\sG)=\emptyset$. Then $\sG$ is a vector bundle with $\det(\sG)\cong \sO_X$. It follows that the 
composed morphism $\sG^* \to \Omega_X^1 \to \sG$ must be an isomorphism. The lemma easily follows.
\hfill $\Box$
\end{proof}

In view of the generalized Bondal conjecture, one may ask the following.

\begin{conj}\label{conjecture:singular_set}
Let $X$ be a complex projective manifold, and let
$\sG \subset T_X$ be foliation of codimension $q$ on $X$ with numerically trivial canonical class.
If the singular locus $Z(\sG)$ of $\sG$
is non-empty, it  contains a component of codimension $ \le q+1$. 
\end{conj}

The statement is obviously true in the case $q=0$.
It is also known to be true if $q=1$ by \cite[Corollary 4.7]{loray_pereira_touzetv3} (see also
Proposition \ref{prop:singular_set_cd1}).

\begin{rem} \rm
In the setup of Conjecture \ref{conjecture:singular_set}, if $Z(\sG)$ is non-empty, then 
$K_X$ is not pseudo-effective by \cite[Theorem 5.4]{loray_pereira_touzetv4}.
\end{rem}

\begin{thm}\label{thm:singular_set}
Let $X$ be a complex projective manifold, and let
$\sG \subset T_X$ be foliation of codimension $ q\le 2$ on $X$ with numerically trivial canonical class.
If the singular locus $Z(\sG)$ of $\sG$
is non-empty, it  contains a component of codimension $\le q+1$. 
\end{thm}

\begin{proof}
Suppose that $Z(\sG)$ is non-empty. Applying \cite[Theorem 5.4]{loray_pereira_touzetv4}, we see that 
$K_X$ is not nef. The theorem then follows from Proposition \ref{prop:singular_set_cd1} and Theorem \ref{thm:main2}.
\hfill $\Box$
\end{proof}

\begin{proof}[Proof of Theorem \ref{thmintro:generalized_bondal_conjecture}]
We maintain notation and assumptions of Theorem \ref{thmintro:generalized_bondal_conjecture}.
The statement is easy if $r = \dim X$. Suppose now that $r \in \{\dim X-2,\dim X-1\}$, and let $\sG$ be
the natural foliation associated to $\tau$. 
Suppose moreover that $X\setminus X_r$ is non-empty and has codimension at least $2$ in $X$.
From Lemma \ref{lemma:singular_set_associated_foliation}, we conclude that $Z(\sG)$ is non-empty as well.
The statement then follows from Theorem \ref{thm:singular_set} above applied to $\sG$ using the fact that $Z(\sG) \subset X\setminus X_r$.
\hfill $\Box$
\end{proof}

\providecommand{\bysame}{\leavevmode\hbox to3em{\hrulefill}\thinspace}
%
%

\bibliographystyle{amsalpha}

\begin{thebibliography}{ABC93}

\bibitem[AD13]{fano_fols}
Carolina Araujo and St{\'e}phane Druel, \emph{On {F}ano foliations}, Adv. Math.
  \textbf{238} (2013), 70--118.
\MR{3033631}

\bibitem[AD14]{codim_1_del_pezzo_fols}
\bysame, \emph{On codimension 1 del {P}ezzo foliations on varieties with mild
  singularities}, Math. Ann. \textbf{360} (2014), no.~3-4, 769--798.
\MR{3273645}

\bibitem[And85]{ando}
Tetsuya Ando, \emph{On extremal rays of the higher-dimensional varieties},
  Invent. Math. \textbf{81} (1985), no.~2, 347--357.
\MR{0799271}

\bibitem[Ati57]{atiyah57}
M.\,F. Atiyah, \emph{Complex analytic connections in fibre bundles}, Trans.
  Amer. Math. Soc. \textbf{85} (1957), 181--207.
\MR{0086359}

\bibitem[AW97]{andreatta_wisniewski_view}
Marco Andreatta and Jaros{\l}aw~A. Wi{\'s}niewski, \emph{A view on contractions
  of higher-dimensional varieties}. In: Algebraic geometry---{S}anta {C}ruz 1995,
  pp.~153--183,
  Proc. Sympos. Pure Math., vol.~62, Amer. Math. Soc., Providence, RI, 1997.
\MR{1492522}

\bibitem[BB70]{baum_bott70}
Paul~F. Baum and Raoul Bott, \emph{On the zeros of meromorphic vector-fields}.
In:
  Essays on {T}opology and {R}elated {T}opics ({M}\'emoires d\'edi\'es \`a
  {G}eorges de {R}ham), pp.~29--47, Springer, New York, 1970, pp.~29--47.
\MR{0261635}

\bibitem[Bea77]{beauville_prym}
Arnaud Beauville, \emph{Vari\'et\'es de {P}rym et jacobiennes
  interm\'ediaires}, Ann. Sci. \'Ecole Norm. Sup. (4) \textbf{10} (1977),
  no.~3, 309--391.
\MR{0472843}

\bibitem[Bea83]{beauville83}
\bysame, \emph{Vari\'et\'es {K}\"ahleriennes dont la premi\`ere classe de
  {C}hern est nulle}, J. Differential Geom. \textbf{18} (1983), no.~4, 755--782
  (1984).
\MR{0730926}

\bibitem[Bea11]{beauville_problem_list}
\bysame, \emph{Holomorphic symplectic geometry: a problem list}. In: Complex and
  differential geometry, pp.~49--63, Springer Proc. Math., vol.~8, Springer, Heidelberg,
  2011. \MR{2964467}

\bibitem[BK08]{balaji_kollar}
V.~{Balaji} and J\'anos {Koll\'ar}, \emph{{Holonomy groups of stable vector
  bundles}}, {Publ. Res. Inst. Math. Sci.} \textbf{44} (2008), no.~2, 183--211.
\MR{2426347}

\bibitem[Bri10]{brion_action}
Michel Brion, \emph{Some basic results on actions of nonaffine algebraic
  groups}. In: Symmetry and spaces, pp.~1--20, Progr. Math., vol. 278, Birkh\"auser Boston,
  Inc., Boston, MA, 2010.
\MR{2562620}

\bibitem[Dru16]{bobo}
St{\'e}phane Druel, \emph{A decomposition theorem for singular spaces with
  trivial canonical class of dimension at most five}, preprint, 2016.
  \href{https://arxiv.org/abs/1606.09006}{arXiv:1606.09006}

\bibitem[Dru17]{druel_fol_fano}
\bysame, \emph{Regular foliations on weak fano manifolds}, Ann. Fac.
  Sci. Toulouse Math. (6) \textbf{26} (2017), no.~1, 207--217.

\bibitem[GP11]{sga3}
Philippe Gille and Patrick Polo (eds.), \emph{Sch\'emas en groupes ({SGA} 3).
  {T}ome {I}. {P}ropri\'et\'es g\'en\'erales des sch\'emas en groupes},
   S{\'e}minaire de
  G{\'e}om{\'e}trie Alg{\'e}brique du Bois Marie 1962--64. A seminar directed by M. Demazure and A.
  Grothendieck with the collaboration of M. Artin, J.-E. Bertin, P. Gabriel, M.
  Raynaud and J-P. Serre. Revised and annotated edition of the 1970 French
  original, Documents Math\'ematiques (Paris), 7, Soci\'et\'e Math\'ematique de France, Paris, 2011.
  \MR{2867621}

\bibitem[HL97]{HuyLehn}
D.~Huybrechts and M.~Lehn, \emph{The geometry of moduli spaces of sheaves},
  Aspects of Mathematics, E31, Friedr. Vieweg \& Sohn, Braunschweig, 1997.
\MR{1450870}

\bibitem[HV10]{hwang_viehweg}
Jun-Muk Hwang and Eckart Viehweg, \emph{Characteristic foliation on a
  hypersurface of general type in a projective symplectic manifold}, Compos.
  Math. \textbf{146} (2010), no.~2, 497--506.
\MR{2601637}

\bibitem[Jou79]{jouanolou}
J.\,P. Jouanolou, \emph{\'{E}quations de {P}faff alg\'ebriques}, Lecture Notes
  in Mathematics, vol. 708, Springer, Berlin, 1979.
\MR{0537038}

\bibitem[{Kem}92]{kempf}
George~R. {Kempf}, \emph{{Pulling back bundles}}, {Pacific J. Math.} \textbf{152}
  (1992), no.~2, 319--322.
\MR{1141798}

\bibitem[KMM87]{kmm}
Yujiro Kawamata, Katsumi Matsuda, and Kenji Matsuki, \emph{Introduction to the
  minimal model problem}. In: Algebraic geometry, {S}endai, 1985, pp.~283--360, Adv. Stud. Pure
  Math., vol.~10, North-Holland, Amsterdam, 1987.
\MR{0946243}

\bibitem[Kol86]{kollar_higher1}
J{\'a}nos Koll{\'a}r, \emph{Higher direct images of dualizing sheaves. {I}},
  Ann. of Math. (2) \textbf{123} (1986), no.~1, 11--42.
\MR{0825838}

\bibitem[{Lan}04]{langer_ss_sheaves}
Adrian {Langer}, \emph{{Semistable sheaves in positive characteristic}}, {Ann.
  Math. (2)} \textbf{159} (2004), no.~1, 251--276.
\MR{2051393}

\bibitem[LPT11a]{loray_pereira_touzetv3}
Frank Loray, Jorge~Vit\'orio Pereira, and Fr\'ed\'eric Touzet, \emph{Singular
  foliations with trivial canonical class}, preprint, 2011. 
\href{https://arxiv.org/abs/1107.1538v3}{arXiv:1107.1538v3}

\bibitem[LPT11b]{loray_pereira_touzetv4}
\bysame, \emph{Singular foliations with trivial canonical class}, preprint, 2015. 
\href{https://arxiv.org/abs/1107.1538v4}{arXiv:1107.1538v4}

\bibitem[Mum08]{mumford_av}
David Mumford,
\emph{Abelian varieties}. With appendices by C. P. Ramanujam and Yuri Manin. Corrected reprint of the second (1974) edition. 
Tata Inst. Fund. Res. Stud. Math., vol. 5.. Published for the Tata Institute of Fundamental Research, Bombay; by Hindustan Book Agency, New Delhi, 2008.
\MR{2514037}

\bibitem[NS65]{narasimhan_seshadri65}
M.\,S. {Narasimhan} and C.\,S. {Seshadri}, \emph{{Stable and unitary vector
  bundles on a compact Riemann surface}}, {Ann. Math. (2)} \textbf{82} (1965),
  540--567.
\MR{0184252}

\bibitem[Pol97]{polishchuk}
A.~Polishchuk, \emph{Algebraic geometry of {P}oisson brackets}, J. Math. Sci.
  (New York) \textbf{84} (1997), no.~5, 1413--1444, Algebraic geometry, 7.
\MR{1465521}

\bibitem[PT13]{pereira_touzet}
Jorge~Vit{\'o}rio Pereira and Fr{\'e}d{\'e}ric Touzet, \emph{Foliations with
  vanishing {C}hern classes}, Bull. Braz. Math. Soc. (N.S.) \textbf{44} (2013),
  no.~4, 731--754.
\MR{3167130}

\bibitem[ST71]{siu_trautmann}
Yum-tong Siu and G\"unther Trautmann, \emph{Gap-sheaves and extension of
  coherent analytic subsheaves}, Lecture Notes in Mathematics, vol. 172,
  Springer-Verlag, Berlin-New York, 1971.
\MR{0287033}

\bibitem[Tou08]{touzet}
Fr{\'e}d{\'e}ric Touzet, \emph{Feuilletages holomorphes de codimension un dont
  la classe canonique est triviale}, Ann. Sci. \'Ec. Norm. Sup\'er. (4)
  \textbf{41} (2008), no.~4, 655--668.
\MR{2489636}

\bibitem[{Tou}10]{touzet_cd2}
Fr\'ed\'eric {Touzet}, \emph{{Structure des feuilletages K\"ahleriens en
  courbure semi-n\'egative}}, {Ann. Fac. Sci. Toulouse, Math. (6)} \textbf{19}
  (2010), no.~3-4, 865--886.
\MR{2790821}

\bibitem[Wi{\'s}91]{wisn_crelle}
J.\,A. Wi{\'s}niewski, \emph{On contractions of extremal rays of {F}ano
  manifolds}, J.~Reine Angew. Math. \textbf{417} (1991), 141--157.
\MR{1103910}
\end{thebibliography}
\bibliographymark{References}

\end{document}